\documentclass[11pt]{article}
\usepackage{amssymb}
\usepackage{amsthm}
\usepackage{mathrsfs}
\usepackage{mathtools}
\usepackage{esint}
\newtheorem{thm}{Theorem}[section]

\newtheorem{lem}[thm]{Lemma}
\newtheorem{prop}[thm]{Proposition}
\theoremstyle{definition}
\newtheorem{defn}[thm]{Definition}
\theoremstyle{remark}
\newtheorem{rem}[thm]{Remark}

\numberwithin{equation}{section}

\def\Xint#1{\mathchoice
	{\XXint\displaystyle\textstyle{#1}}%
	{\XXint\textstyle\scriptstyle{#1}}%
	{\XXint\scriptstyle\scriptscriptstyle{#1}}%
	{\XXint\scriptscriptstyle\scriptscriptstyle{#1}}%
	\!\int}
\def\XXint#1#2#3{{\setbox0=\hbox{$#1{#2#3}{\int}$}
		\vcenter{\hbox{$#2#3$}}\kern-.5\wd0}}

\def\dashint{\Xint-}
\usepackage{color}
\usepackage[margin=2cm]{geometry}

\begin{document}
	
	%
	%
	%
	%
	%
	%
	%
	%
	%
	
	
	\title{Homogenization and integral representation of energy functionals in manifold valued Orlicz-Sobolev spaces}
	\author{Joseph Dongho,\footnote{University of Maroua, Department of Mathematics and Computer Science, P.O. box 814, Maroua, Cameroon email:joseph.dongho@fs.univ-maroua.cm, postumus} \,
		Joel Fotso Tachago,\footnote{University of Bamenda, 	Higher Teachers Training College, Department of Mathematics
			P.O. Box 39
			Bambili, Cameroon, email:fotsotachago@yahoo.fr}  \\
		Franck Tchinda,\footnote{University of Bertoua, Department of Mathematics, Statistics and Computer Science, P.O. Box 652, Bertoua, Cameroon email:takougoumfranckarnold@gmail.com} \, 	and Elvira
		Zappale\footnote{Department of Basic and Applied Sciences for Engineering, Sapienza-  University of Rome, via
			A. Scarpa,  16 (00161) Roma, Italy, e-mail:
			elvira.zappale@uniroma1.it}}\date{ } \maketitle
	\begin{abstract} 
	This paper aims to extend to Orlicz-Sobolev spaces some results of integral representation for the simultaneous homogenization and dimensional reduction of integral energies defined on fields taking values on a differentiable manifold. Since our functional framework goes beyond the classical Sobolev's spaces, we also prove, via $\Gamma$-convergence, a general integral representation results in the unconstrained Orlicz setting. 
        Due to $\Delta_2$ and $\nabla_2$ conditions verified by the Young function $\Phi$ (which modulated the growth behaviour), we prove that the density of the $\Gamma$-limit is a tangential quasiconvex integrand represented by a cell formula.  

	\end{abstract}

	\maketitle
	
	
	\section{Introduction} \label{sect1} 

The so called homogenization theory consists of detecting the macroscopic overall behaviour of a model (described by means either of
partial differential equations or energy functional) with heterogeneous coefficients
that (periodically) oscillate on a smaller scale, say $\varepsilon$. In order to rigorously derive  properties by means of a limiting procedure as the fine-scale $\varepsilon$ converges to zero, many tools have been developed in the last decades ranging from
asymptotic expansion methods (e.g. see A. E. Sanchez-Palencia \cite{SP80} and the references therein) or the H-convergence methods due to
F. Murat and L. Tartar \cite{Tar77, Tar09, FMT09} or the two-scale convergence \cite{Ngu89, All2}, more recently re-casted in terms of a fixed functional space, i.e. within the theory of periodic unfolding (see \cite{CDGbook} for an exhaustive treatment).

From the variational stand-point a crucial tool is De Giorgi's $\Gamma$-convergence (see \cite{DMbook} for details and results).

A rich area of research, due to the many application, is the derivation of
lower dimensional theories, i.e. the modeling of thin structures in elasticity and micromagnetics. In this context, one is interested in detecting a reduced model, via  asymptotic analysis departing from a slender body,  usually in $3D$,  
letting the geometry of the body become singular in one or more directions.
Many results have been proven in this context again exploiting
 $\Gamma$-convergence as in the pioneering papers \cite{ABP91}, \cite{LDR95} (dealing with the pure elastic setting) or in \cite{GJ}, \cite{AL01}, \cite{P}, \cite{BZ07}, \cite{GH, GH2} (in the micromagnetic and ferromagnetic framework), within a wider literature. 
  We refer to \cite{FIOM14}, \cite{KK16}, \cite{BV17}, \cite{FG23}, \cite{CGO24}, among a vast bibliography for simultaneous dimension reduction and homogenization in the realm of nonlinear heterogeneous thin structures and composites with standard growth or with the two-scale convergence technique and for plates and rods (\cite{Neukamm}). On the other hand, the same procedure has not been taken into account in the constrained setting, 
 
 Very recently in \cite{ELTZ} the simultaneous dimension reduction and homogenization has been obtained with micromagnetic type constraint, within the classical Sobolev setting and in \cite{LTZ} in the linear growth case, i.e. for manifold valued $BV$-maps. 
 
 In this paper we aim at extending the results in \cite{ELTZ} casting our integral functionals in the Orlicz framework, thus allowing for more general growth conditions. 
 
 In details, in order to model problems dealing with liquid crystals, magnetostrictive or ferromagnetic materials, we consider here a slender domain approaching a reduced one as the heterogeneity becomes finely and finely distributed, i.e. we consider homogenization and dimensional reduction to happen at the same {\it speed}. This analysis, despite, in principle very similar to the one with classical power growth-type energies, in \cite{ELTZ}, cannot rely on the same exact tools. In fact, first  we need to prove more general integral representation results in the unconstrained setting, i.e. we extend to the Sobolev-Orlicz frame, some of the results contained in \cite{BFF}, cf. Theorems \ref{thm2.5BFF} and \ref{intrep} below. In particular this result, relying on more general functions spaces properties and abstract integral representation theorems (see \cite{AC2005} and \cite{MM}) constitutes a generalization of the particular cases treated in \cite{LN, Elvira3} in the dimensional reduction context and of the homogenization results in \cite{tacha2, tacha5, nguet3, tacha7, tacha1,  tacha6, tacha3}.

 At this point it is also worth to underline that 
our results extend to the Orlicz framework the theorems obtained in \cite{baba2} for the pure homogenization manifold valued setting homogeneous Orlicz-Sobolev and those proven in \cite{DFMT}
in the surface valued relaxation and integral representation, see Theorems \ref{c1eq3Ma} and \ref{c1eq6Ma} in Subsection \ref{Corollaries}.


We are now in position to present the model we are focusing on. We assume that our domain is an inhomogeneous cylinder, whose microstructure is distributed with periodicity within the material and this is described by the small real parameter $\varepsilon >0$ comparable with the height of the domain. The equilibrium configurations are the (almost) minimizers of an
integral functional taking the form 
\begin{equation}\label{funct1}
\int_{\omega_\varepsilon } f\left ( \frac{x}{\varepsilon}, \nabla u\right ) \, dx \qquad u: \omega_\varepsilon  \rightarrow \mathbb{R}^3,
\end{equation}
subject to suitable boundary conditions, and where $\omega \subset \mathbb{R}^2$ is a bounded open set, $\omega_\varepsilon:= \omega\times (-\frac{\varepsilon}{2},\frac{\varepsilon}{2})$ and $f: \mathbb{R}^3 \times \mathbb{R}^{3 \times 3}\rightarrow \mathbb{R}$
is a periodic integrand with respect to the first variable, and $u$ is a manifold-valued Orlicz-Sobolev field, that will be specialized in the sequel. 


We also assume that $f$ has suitable growth and coercivity in the gradient variable, stated in terms of the Young ($N$)-function $\Phi$ which describes the Orlicz space. 

Our contribution consists in applying both homogenization and dimension reduction simultaneously, assuming the functional defined on the space of manifold-valued Orlicz-Sobolev fields.


We consider the model $3D-2D$ but our analysis could be extended to other dimensions, i.e. to the framework $ND- (N-d)D$. 

More precisely, we focus on the energy in \eqref{funct1} assuming that
$f:\mathbb R^3 \times\mathbb R^{3\times 3}\to \mathbb{R}$ is a Carath\'eodory function, and
$\mathcal{M}$ is a smooth submanifold of $\mathbb{R}^3$ without boundary. In particular, we assume that $f$ has the following properties:
\begin{itemize}
    \item [(H1)] $f(\cdot, x_3, \xi)$ is 1-periodic, i.e. for $\mathcal L^3$-a.e. $(x_{\alpha}, x_3) \in \mathbb{R}^3$ and $\xi \in \mathbb{R}^{3 \times 3}$ it holds
    \begin{equation*}
        f(x_{\alpha} + \mathbf{e}_i, x_3, \xi) = f(x_{\alpha}, x_3, \xi), \qquad \forall i = 1,2
    \end{equation*}
    where $\{\mathbf{e}_1, \mathbf{e}_2\}$ is the canonical basis of $\mathbb{R}^2.$

    \item [(H2)] 
There exists a  Young function $\Phi$ of class $\Delta_{2}\cap\nabla_{2}$ (see \eqref{Delta2alpha} and \eqref{nabla2beta} in Section \ref{sect2M}), such that, for almost all $y\equiv(y_\alpha, y_3) \in \mathbb{R}^{3}$, $f(y, \cdot)$ satisfies $\Phi$-growth and $\Phi$-coercivity conditions, i.e., there exist $0< \alpha \leq \beta < +\infty$ such that: 
		\begin{equation*}
			\alpha \Phi(|\xi|) \leq f(y, \xi) \leq \beta (1+\Phi(|\xi|)),
		\end{equation*}
    for a.e. $y\in\mathbb R^3$ and for every $\xi\in\mathbb R^{3\times 3}$.
\end{itemize}

Referring to subsection \ref{sect2M} for detailed notation, we define the functional $\tilde{I}^\varepsilon: L^\Phi(\omega_{\varepsilon};\mathcal{M}) \rightarrow \overline{\mathbb{R}}$
\[
\tilde{I}^\varepsilon(u) := \left \{
\begin{array}{lll}
\!\!\!\!\!\! & \displaystyle \frac{1}{\varepsilon}\int_{\omega_{\varepsilon}} f \left (\frac{x}{\varepsilon}, \nabla u\right ) \, dx \qquad & \textnormal{if }u \in W^{1}L^{\Phi}(\omega_{\varepsilon}; \mathcal{M})\\[4mm]
\!\!\!\!\!\! & + \infty &\textnormal{elsewhere.}
\end{array}
\right.
\]
The study of the $\Gamma-$limit of $\tilde{I}^\varepsilon$ is equivalent to the study of the $\Gamma-$limit of the rescaled functional $I^\varepsilon$ defined as
\begin{equation}
\label{functional}
    I^\varepsilon(u) := \left \{
\begin{array}{lll}
\!\!\!\!\!\! & \displaystyle \int_{\Omega} f \left (\frac{x_{\alpha}}{\varepsilon}, x_3, \nabla_\varepsilon u\right ) \, dx \qquad & \textnormal{if $u \in W^{1}L^{\Phi}(\Omega; \mathcal{M})$}\\[4mm]
\!\!\!\!\!\! & + \infty &\textnormal{elsewhere}
\end{array}
\right.
\end{equation}
with $\Omega := \omega \times \left (- \frac{1}{2}, \frac{1}{2} \right ) = \omega_{1}$, $x_\alpha:=(x_1,x_2)\in \omega$ and $\nabla_\varepsilon:=[\frac{\partial}{\partial x_1}, \frac{\partial}{\partial x_2}, \frac{1}{\varepsilon}\frac{\partial}{\partial x_3}]$.
From here onward, we denote $\nabla_\alpha:=[\frac{\partial}{\partial x_1}| \frac{\partial}{\partial x_2}]$ and $\nabla_3:=\frac{\partial}{\partial x_3}$, and $\nabla_\varepsilon:=[\nabla_\alpha| \frac{1}{\varepsilon}\nabla_3]$. Moreover, we denote with $\xi_\alpha:=[\xi_1 | \xi_2]$ an element of $\mathbb R^{3\times 2}$ and with $\xi:=[\xi_\alpha|\xi_3]$ an element of $\mathbb R^{3\times 3}$, with $\xi_i$ a column vector in  $\mathbb R^3$, $i=1,\dots, 3$.\\

\noindent
Our main result is the following, which, having in mind concrete applications, is stated in dimension three, but the result holds in generic dimensions $N$ and $d$ (domain and range dimensions, respectively).

\begin{thm}
\label{casopm1}
    Let $\Phi$ be a Young function of class $\Delta_{2}\cap\nabla_{2}$. Assume that $\mathcal{M}$ is a connected smooth manifold of $\mathbb{R}^3$ without boundary and let $f:\mathbb R^3 \times\mathbb R^{3\times 3}\to \mathbb{R}$ be a Carathéodory function satisfying {\rm (H1)} and {\rm (H2)}. Let $\omega\subset \mathbb R^2$ be a bounded open set with Lipschitz boundary. Let $\Omega:=\omega \times(-1/2,1/2)$ and for every $\varepsilon \in \mathbb R^+$, let $I_\varepsilon$ be as in \eqref{functional}. Then, the $\Gamma-$limit of $I^\varepsilon$ as $\varepsilon \rightarrow 0$ with respect to the strong $L^\Phi(\Omega)$-topology is the functional $I: L^\Phi(\omega; \mathcal{M}) \rightarrow \overline{\mathbb{R}}$ given by 
\begin{equation}
\label{candidatepm1}
    I(u) = \left \{
\begin{array}{lll}
\!\!\!\!\!\! & \displaystyle\int_{\omega} T f_{\rm hom}^0(u, \nabla_\alpha u) \, dx_\alpha, & \textnormal{if $u \in W^{1}L^{\Phi}(\omega; \mathcal{M})$}\\[4mm]
\!\!\!\!\!\! & + \infty &\textnormal{elsewhere,}
\end{array}
\right.
\end{equation}
with $T f^0_{\rm hom}: \mathbb{R}^3 \times \mathbb{R}^{3 \times 2} \rightarrow \mathbb{R}$ defined as
\begin{align}
\label{Tfhom}
&Tf^0_{\rm hom}(s,\xi_\alpha) :=  \liminf_{t \rightarrow + \infty} \inf_{\varphi} \Bigg \{\frac{1}{t^2} \int_{(tQ')_{,1}} f(x_{\alpha}, x_3,\xi_\alpha + \nabla_{\alpha} \varphi | \nabla_3 \varphi) \, d x_{\alpha} d x_3: \nonumber\\
& \, \varphi \in W^{1,\infty}((tQ')_{,1}; T_s (\mathcal{M})), \,\, \varphi(x_{\alpha}, x_3) = 0\,\,\, \textnormal{for every $(x_{\alpha}, x_3) \in \partial (tQ') \times \left(-\frac{1}{2},\frac{1}{2}\right)$} \Bigg \},
\end{align}
where $Q':=\left(-\frac{1}{2},\frac{1}{2}\right)^2,$ $(tQ')_{,1}:=tQ'\times \left(-\frac{1}{2},\frac{1}{2}\right), $ and $T_s (\mathcal{M})$ denotes the tangent space to $\mathcal{M}$ in $s$.
\end{thm}

\color{black}

    At this point, for a better understanding of the above statement, it is worth to observe that
		for a smooth $\mathcal{M}$-valued map $u$, it is well known that the first order derivatives belong to the tangent space of $\mathcal{M}$. More precisely, for any open set $U\subset \mathbb R^N$, with Lipschitz boundary, $u \in 	W^{1}L^{\Phi}(U; \mathcal{M})$, this property still holds in the sense that $\nabla u(x) \in [T_{u(x)}(\mathcal{M})]^{N}$ for $\mathcal{L}^{N}-\mbox{a.e. } x\in U$.

	\subsection{Corollaries}\label{Corollaries}

    Taking into account that Theorem \ref{casopm1} can be proved for any $N$ and $d \in \mathbb N$, assuming that no dimensional reduction is taken into account, i.e. $\omega\subset \mathbb R^N$ is a generic open set with Lipschitz boundary, then the following result holds in the periodic setting, thus leading to the attainment of a homogenization result in the Sobolev-Orlicz setting which generalizes the result of  J.-F. Babadjian and V. Millot in \cite{baba2}.
    
	\begin{thm}\label{c1eq3Ma}
		Let $\Phi$ be a Young function of class $\Delta_{2}\cap\nabla_{2}$, $\mathcal{M}$ be a connected smooth submanifold of $\mathbb{R}^{N}$ without boundary and $f : \mathbb{R}^{N}\times\mathbb{R}^{d\times N} \to [0,\infty)$  be a Carathéodory function such that
        \begin{itemize}
    \item [(H1')] $f(\cdot, \xi)$ is 1-periodic, i.e. for $\mathcal L^N$-a.e. $x \in \mathbb{R}^N$ and $\xi \in \mathbb{R}^{d \times N}$ it holds
    \begin{equation*}
    \label{H1}
        f(x + \mathbf{e}_i, \xi) = f(x, \xi), \qquad \forall i = 1,\dots, N,
    \end{equation*}
    where $\{\mathbf{e}_1, \dots, \mathbf{e}_N\}$ is the canonical basis of $\mathbb{R}^N;$

    \item [(H2')] 
    for $\mathcal L^N$-a.e.$y \in \mathbb{R}^{N}$, $f(y, \cdot)$ satisfies $\Phi$-growth and $\Phi$-coercivity conditions, i.e., there exist $0< \alpha \leq \beta < +\infty$ such that: 
		\begin{equation*}
			\alpha \Phi(|\xi|) \leq f(y, \xi) \leq \beta (1+\Phi(|\xi|)),
		\end{equation*}
    for a.e. $y\in\mathbb R^N$ and for every $\xi\in\mathbb R^{d\times N}$.
\end{itemize}
Then the family of functionals $(\mathcal{F}_{\varepsilon})_{\varepsilon>0}$ defined as
\begin{equation}
\label{functionalFe}
    \mathcal F_\varepsilon(u) := \left \{
\begin{array}{lll}
\!\!\!\!\!\! & \displaystyle \int_{\omega} f \left (\frac{x}{\varepsilon}, \nabla u\right ) \, dx \qquad & \textnormal{if $u \in W^{1}L^{\Phi}(\omega; \mathcal{M})$}\\[4mm]
\!\!\!\!\!\! & + \infty &\textnormal{elsewhere}
\end{array}
\right.
\end{equation}$\Gamma$-converges for the strong $L^{\Phi}(\omega)$-topology to the functional $\mathcal{F}_{\textup{hom}} : L^{\Phi}(\Omega; \mathbb{R}^{d}) \to [0,\infty)$ defined by 
		\begin{equation*}
			\mathcal{F}_{\textup{hom}}(u) := \left\{\begin{array}{ll}
				\int_{\omega} Tf_{\textup{hom}}\left(u, \nabla u\right) dx & \mbox{if } u \in W^{1}L^{\Phi}(\omega; \mathcal{M}),  \\
				+\infty &  \mbox{otherwise},
			\end{array}\right.	
		\end{equation*}
		where, for all $s\in \mathcal{M}$ and $\xi \in [T_{s}(\mathcal{M})]^{N}$,
		\begin{equation*}
			Tf_{\textup{hom}}(s, \xi) := \lim_{t\to +\infty} \inf_{\varphi} \left\{\dashint_{(0,t)^{N}} f(y, \xi + \nabla \varphi(y)) dy\, : \, \varphi \in W^{1}_{0}L^{\Phi}((0,t)^{N}; T_{s}(\mathcal{M})) \right\}
		\end{equation*}
		is the tangentially homogenized energy density.
	\end{thm}
	Secondly, Theorem \ref{casopm1} (clearly also Theorem \ref{c1eq3Ma}) leads to the following corollary in the homogeneous setting, i.e. if one considers the integrand $f$ in \eqref{functionalFe} not depending on the first variable, it is possible to extend to the Orlicz setting the  relaxation result obtained in \cite{DFMT}.
	\begin{thm}\label{c1eq6Ma}
		Let $\Phi$ be a Young function of class $\Delta_{2}\cap\nabla_{2}$, $\mathcal{M}$ be a connected smooth submanifold of $\mathbb{R}^{d}$ without boundary. Let $\omega\subset \mathbb R^N$ be a bounded open set with Lipschitz boundary and $g : \mathbb{R}^{d\times N} \to [0,\infty)$  be a continuous function satisfying $(H_{2}')$. Then the functional $\mathcal G: W^{1}L^\Phi(\omega;\mathbb R^d) \to [0,+\infty)$ defined by 
		\begin{eqnarray}
			\mathcal{G}(u) := \inf_{\{u_{n}\}} \bigg\{\liminf_{n\to\infty} \int_{\Omega} g(\nabla u_{n}) dx \, :\, u_{n} \rightharpoonup u\, \mbox{in }W^{1}L^{\Phi}(\omega; \mathbb{R}^{d})\mbox{-weakly}, 
			u_{n}(x) \in \mathcal{M}, \, \mbox{a.e. } x\in \omega \bigg\}  \nonumber
		\end{eqnarray}
		coincides with the functional $J: W^{1}L^\Phi(\omega;\mathbb R^d) \to [0,+\infty)$, defined as  
		\begin{equation*}
			\mathcal{J}(u) := \int_{\omega} \mathcal{Q}_{T}g(u, \nabla u) dx,
		\end{equation*}
		where the tangential quasiconvexification $\mathcal{Q}_{T}g$ of $g$ is defined, for all $s\in\mathcal{M}$ and $\xi \in [T_{s}(\mathcal{M})]^{N}$, by 
		\begin{equation}\label{c1eq9Ma}
			\mathcal{Q}_{T}g(s, \xi) :=  \inf \left\{\dashint_{Q} g(\xi + \nabla \varphi(y)) dy\, : \, \varphi \in \mathcal{C}_{c}^{\infty}(Q; T_{s}(\mathcal{M})) \right\}.
		\end{equation}
	\end{thm}
	\begin{rem}
		\begin{itemize}
        \item Theorem \ref{casopm1} and (its two corollaries above stated) are an extension to Orlicz-Sobolev spaces of \cite[Theorem 1.1]{ELTZ}. 
        
       \item Note that, as in \cite{tacha7}, \eqref{c1eq9Ma} remains true if we replace $\mathcal{C}_{c}^{\infty}(Q; T_{s}(\mathcal{M}))$ by $W^{1}_{0}L^{\Phi}(Q; T_{s}(\mathcal{M}))$.
       \end{itemize}
	\end{rem}
	\color{black}

	The paper is divided into sections each revolving around a specific aspect: Section \ref{sect2M} dwells on notation and preliminary results.  
    Section \ref{mainsect} is devoted to the simultaneous dimension reduction and homogenization in the unconstrained setting.  
    Section \ref{subsect2M} deals with the energy densities appearing in Theorems \ref{casopm1}, \ref{c1eq3Ma} and \ref{c1eq6Ma}, i.e. the properties of the energy density $Tf^0_{\textup{hom}}$ (as a byproduct also for $Tf_{\textup{hom}}$ and $\mathcal{Q}_{T}g$).
    
    The proof of Theorem \ref{casopm1} is then contained in Section \ref{sect3M}. 
	\color{black}
    \section{Notation and preliminary results }\label{sect2M}

	
	\begin{itemize}
		\item Otherwise unless differently specified, $\omega$ is an open bounded subset of $\mathbb{R}^{2}$ with Lipschitz boundary.
		\item For any given open set $U \subset \mathbb R^N$, $\mathcal{A}(U)$ denotes the family of all open subsets of $U$.
        \item For every open sets $C$ and $A$, $C \Subset A$, means that $\overline C \subset A$;
		\item $B^{N}(s, r)$ is the closed ball in $\mathbb{R}^{N}$ of center $s \in \mathbb{R}^{N}$ and radius $r>0$, when the superscript $N$ is omitted, $B(s,r)$ will denote an open ball of $\mathbb R^2$.
		\item  $\mathcal{L}^{N}$ is the Lebesgue measure in $\mathbb{R}^{N}$.
		\item  $\mathbb{R}^{d\times N}$ is identified with the set of real $d\times N$ matrices.
		\item $Q := (-\frac{1}{2},\frac{1}{2})^{N}$ is the unit cube in $\mathbb{R}^{N}$ and $Q(x_{0},\rho) := x_{0} + \rho Q$ with $x_{0} \in \mathbb{R}^{N}$, $\rho >0$. 
        \item  We also denote by $Q'$ the square $(-\frac{1}{2}, \frac{1}{2})^2$. 
       
		\item The symbol $\dashint_{A}$ stands for the average $\mathcal{L}^{N}(A)^{-1}\int_{A}$.
 \item Given a set $A\subset\omega$ we denote by $A_{,\varepsilon}$, with $\varepsilon>0$, the set $A\times (-\frac{\varepsilon}{2}, \frac{\varepsilon}{2})$ and so $A_{,1}:=A\times(-\frac{1}{2}, \frac{1}{2})$. In particular, $\Omega=\omega_{,1}$.
\item  $Q'_{,\varepsilon}=(-\frac{1}{2}, \frac{1}{2})^2\times (-\frac{\varepsilon}{2}, \frac{\varepsilon}{2})$ for $\varepsilon>0$, while $Q'_{,1}=(-\frac{1}{2}, \frac{1}{2})^3=Q$. 
        \item  $\mathcal{C}_{c}(U)$  is the space of continuous  functions $f: U\to \mathbb{R}$ with compact support.
		\item  $\mathcal{C}_{0}(U)$ is the closure of $\mathcal{C}_{c}(U)$ for the uniform convergence, it coincides with the space of all continuous functions $f: U\to \mathbb{R}$ such that, for every $\eta>0$, there exists a compact set $K_{\eta}\subset U$ with $|f|< \eta$ on $\Omega\backslash K_{\eta}$.
		\item  $\mathcal{M}(U)$ is the space of real-valued Radon measures with finite total variation. We recall that by the Riesz Representation Theorem $\mathcal{M}(U)$ can be identified with the dual space of $\mathcal{C}_{0}(U)$ through the duality 
		\begin{equation*}
			\langle \mu, \phi\rangle = \int_U \phi d\mu, \quad \mu \in \mathcal{M}(U), \;\; \phi \in \mathcal{C}_{0}(U).
		\end{equation*}
		\item  $\mathcal{M}^{+}(U)$ is the subset of $\mathcal{M}(U)$ of positive measures.
		\item If $\mu \in \mathcal{M}(U)$ and $\lambda \in \mathcal{M}^{+}(U)$, we denote by $\frac{d \mu}{d \lambda}$ the Radon-Nikod\'{y}m derivative of $\mu$ with respect to $\lambda$, and by \cite[Proposition 2.2]{AFP} there exists a Borel set $E$ such that $\lambda(E) = 0$ and 
		\begin{equation*}
			\dfrac{d \mu}{d \lambda}(x) = \lim_{\rho\to 0^{+}} \dfrac{\mu(Q(x,\rho))}{\lambda(Q(x,\rho))}
		\end{equation*}
		for all $x\in \mbox{Supp }\mu \, \backslash\, E$.
		\item $\Phi: [0,\infty)\to[0,\infty)$ is a Young function, i.e. $\Phi$ is continuous, convex, with $\Phi(t)>0$ for $t>0$, $\frac{\Phi(t)}{t}\to 0$ as $t\to 0$, and $\frac{\Phi(t)}{t}\to \infty$ as $t\to \infty$.
		\item $\widetilde{\Phi}$ stands for the complementary of Young function $\Phi$, defined by 
		\begin{equation*}
			\widetilde{\Phi}(t)= \sup_{s\geq 0} \big\{st - \Phi(s), \; t\geq 0 \big\}.
		\end{equation*} 
		\item We recall that a Young function $\Phi$ is of class $\Delta_{2}$ at $\infty$ (denoted $\Phi\in \Delta_{2}$) if there are $\alpha>0$ and $t_{0}\geq 0$ such that 
		\begin{equation}\label{Delta2alpha}
			\Phi(2t) \leq \alpha \Phi(t),\; \textup{for\,all}\, t\geq t_{0}.
		\end{equation}
		Also $\Phi$ is of class $\nabla_2$ if $\tilde \Phi$ is of class $\Delta_2$, i.e. $\exists \beta >0$ and $t_0 >1$ such that  
		\begin{equation}\label{nabla2beta}\Phi(t) \leq \frac{1}{2 \beta} \Phi(\beta t), \hbox{ for all }t \geq t_0.
		\end{equation}
		
		\item $L^{\Phi}(U;\mathbb{R}^{d})$ is the Orlicz space of functions defined by 
		\begin{equation*}
			L^{\Phi}(\Omega;\mathbb{R}^{d}) = \left\{ u: \Omega\to \mathbb{R}^{d}\,;\, u\,\textup{is\,measurable},\; \lim_{\alpha\to 0}\int_{U}\Phi(\alpha|u(x)|)dx = 0 \right\}.
		\end{equation*}
		We recall that $L^{\Phi}(U;\mathbb{R}^{d})$ is a Banach space with respect to the Luxemburg norm 
		\begin{equation*}
			\|u\|_{\Phi} = \inf \left\{ k>0\,:\, \int_{U}\Phi\left(\dfrac{|u(x)|}{k}\right)dx \leq 1 \right\}\, < \, +\infty.
		\end{equation*}
		We will denote the norm of elements in $L^{\Phi}(U;\mathbb{R}^{d})$, both by  $\|\cdot\|_{\Phi}$ and $\|\cdot\|_{L^{\Phi}}$.
		\item   $\mathcal{D}(U)$ (or $C^\infty_c(U)$) is the space of  indefinitely differentiable functions $f: U\to \mathbb{R}^{d}$ with compact support. We recall that $\mathcal{D}(U)$ is dense in $L^{\Phi}(U;\mathbb{R}^{d})$, $L^{\Phi}(U;\mathbb{R}^{d})$ is separable and reflexive when $\Phi, \in \Delta_{2}\cap \nabla_2$.  The dual of $L^{\Phi}(U;\mathbb{R}^{d})$ is identified with $L^{\widetilde{\Phi}}(U;\mathbb{R}^{d})$. The property $\frac{\Phi(t)}{t}\to \infty$ as $t\to \infty$ implies that 
		\begin{equation*}
			L^{\Phi}(U;\mathbb{R}^{d}) \subset L^{1}(U;\mathbb{R}^{d}) \subset L^{1}_{loc}(U;\mathbb{R}^{d}) \subset \mathcal{D}'(U),
		\end{equation*}
		each embedding being continuous.
		\item $W^{1}L^{\Phi}(U,\mathbb{R}^{d})$ is the Orlicz-Sobolev space defined by
		\begin{equation*}
			W^{1}L^{\Phi}(U,\mathbb{R}^{d}) = \left\{ u \in L^{\Phi}(U;\mathbb{R}^{d})\,;\, \dfrac{\partial u}{\partial x_{i}} \in L^{\Phi}(U;\mathbb{R}^{d}),\, 1\leq i\leq N \right\},
		\end{equation*}
		where derivatives are taken in the distributional sense on $U$. Endowed with the norm 
		\begin{equation*}
			\|u\|_{W^{1}L^{\Phi}} = \|u\|_{L^{\Phi}} + \sum_{i=1}^{N} \left\|\dfrac{\partial u}{\partial x_{i}} \right\|_{L^{\Phi}}, \; u \in W^{1}L^{\Phi}(U,\mathbb{R}^{d}),
		\end{equation*}
		$W^{1}L^{\Phi}(U,\mathbb{R}^{d})$ is a reflexive Banach space when $\Phi \in \Delta_{2}\cap \nabla_2$. 
		\item $W^{1}_{0}L^{\Phi}(U,\mathbb{R}^{d})$ denotes the closure of $\mathcal{D}(U)$ in $W^{1}L^{\Phi}(U,\mathbb{R}^{d})$ and the semi-norm 
		\begin{equation*}
			u \longrightarrow \|u\|_{W^{1}_{0}L^{\Phi}} = \|Du\|_{L^{\Phi}} =  \sum_{i=1}^{N} \left\|\dfrac{\partial u}{\partial x_{i}} \right\|_{L^{\Phi}}
		\end{equation*}
		is a norm on $W^{1}_{0}L^{\Phi}(U,\mathbb{R}^{d})$ equivalent to $\|\cdot\|_{W^{1}L^{\Phi}}$.
        For other results concerning functional properties of the above spaces we refer to \cite{tacha9} and the bibliography contained therein.
        
	\end{itemize}

  We also recall the following result (see \cite[Theorem 2.4]{focar1})
\begin{thm}\label{thmcompact}
Let $\Omega$ be a bounded open set with Lisphitz boundary and let $\Phi$ be an Orlicz function satisfying $\Delta_2$ condition, then  the embedding
$W^1L^\Phi(\Omega)\hookrightarrow L^\phi(\Omega)$ is compact.
\end{thm}

    \subsection{$\Gamma$-convergence}
We recall the definition and main properties of $\Gamma$-convergence. For a deeper overview of this topic, we refer to \cite{DMbook}. 
The following is a definition in the metric setting, which is sufficient for our purposes. 
\begin{defn}[\cite{DMbook}, Proposition $8.1$]
    Let $(X,d)$ be a metric space. $\forall k \in \mathbb N$ let  $F_k: X \rightarrow \mathbb R \cup \{+\infty\}$, be a functional. Then $\{ F_k\}_k$ $\Gamma$-converges to $F: X \rightarrow \mathbb R \cup \{+\infty\}$ if 
    \begin{itemize}
        \item[$(i)$] ($\Gamma$-liminf inequality) for every $ x \in X$ and every sequence $\{x_k\}$ such that $ d(x_k,x) \to 0$ as $k\to +\infty$, it is
        \[ F(x) \le \liminf_{k \rightarrow \infty} F_k(x_k); \]
        \item[$(ii)$]($\Gamma$-limsup inequality: existence of a recovery sequence) for every $x \in X$, there exists a sequence $\{x_k\}$ such that $d(x_k,x)\to 0$ as $k\to +\infty$, such that
        \[ F(x) = \lim_{k \rightarrow \infty} F_k(x_k).\]
    \end{itemize}
\end{defn}

In fact, we write
\[
F(x):=\inf\{\liminf F_k(x_k): d(x_k,x)\to 0, \hbox{ as }k \to +\infty\}.
\]
A fundamental property we want to underline is that the $\Gamma$-limit is lower semi-continuous with respect to the convergence induced by $d$, see \cite[Proposition 6.8]{DMbook}.

We also provide the definition for $\Gamma$-convergence for a family of functionals.
\begin{defn}\label{defGammafamilies}
    Let $(X,d)$ be a metric space. For a positive parameter $\varepsilon$, we say that a family $\{ F_{\varepsilon}\}_{\varepsilon}$ of functionals, with $F_\varepsilon: X \rightarrow \mathbb R \cup \{+\infty\}, \Gamma$-converges to $F: X \rightarrow \mathbb R \cup \{+\infty\}$, with respect to the metric $d$ as $\varepsilon \rightarrow 0^+$, if for all vanishing sequences $\{\varepsilon_k\}$,  $\{F_{\varepsilon_k}\}_k $ $\Gamma$-converges to $F: X \rightarrow \mathbb R \cup \{+\infty\}$, when $k \rightarrow \infty$.
\end{defn}
We will write, as for the case of sequences,
\[
F(x):=\inf\{\liminf F_\varepsilon(x_\varepsilon): d(x_\varepsilon,x)\to 0, \hbox{ as }\varepsilon \to 0\}.
\]





	\section{Dimension reduction and homogenization in the unconstrained Sobolev-Orlicz setting}\label{mainsect}
    
    This section is devoted to the proof of integral representation results for the $\Gamma$-limit of generic integral functionals in $W^{1}L^\Phi(\Omega;\mathbb R^d)$ defined in thin $3D$ domains, $\Phi$ being a Young functions satisfying $\Delta_2$ and $\nabla_2$ conditions. 

Adopting the same notation as in Section 1, for $\Omega$ and the rescaled gradients, assume that, for every $\varepsilon>0$, the functions $W_\varepsilon: \Omega \times \mathbb R^{3\times 3}\to \mathbb R$ are Carath\'eodory ones satisfying
\begin{equation}\label{growthW}
 \beta' \Phi(\xi)-\frac{1}{\beta'}\leq W_\varepsilon(x_\alpha, x_3, \xi) \leq \beta (1+\Phi(\xi))
\end{equation}
for certain $0<\beta'<\beta$ and every $\xi \in \mathbb R^{3\times 3}$.

Introduce the functionals
    \begin{equation}
\label{functionalfree}
    E_\varepsilon(u) := \left \{
\begin{array}{lll}
\!\!\!\!\!\! & \displaystyle \int_{\Omega} W_\varepsilon \left (x_{\alpha}, x_3, \nabla_\varepsilon u\right ) \, dx \qquad & \textnormal{if $u \in W^{1}L^{\Phi}(\Omega; \mathbb R^3)$}\\[4mm]
\!\!\!\!\!\! & + \infty &\textnormal{elsewhere}
\end{array}
\right.
\end{equation}

    We can define, for any sequence $\{\varepsilon\}$
    \begin{align}\label{freeGammaliminf}
    J_{\{\varepsilon\}}(u;A):=\inf\left\{\liminf_{\varepsilon \to 0}E_\varepsilon(u_\varepsilon): W^1L^\Phi(\Omega;\mathbb R^3) \ni u_\varepsilon \to u \hbox{ in }L^\Phi(\Omega;\mathbb R^3)\right\}.
    \end{align}

    We start observing that, arguing as in the classical $p$-growth case, exploiting \eqref{growthW}, one can prove that energy bounded sequences $\{u_\varepsilon\}_\varepsilon$, in view of Theorem \ref{thmcompact} converge, up to a not relabeled subsequence, to $u \in W^1L^\Phi(\Omega;\mathbb R^3)$, such that $\partial_3 u=0$, i.e. $u$ can be identified with an element, still denoted in the same way, with an abuse of notation, $u \in W^1L^\Phi(\omega;\mathbb R^d)$. 
Then the following integral representation holds. 
\begin{thm}\label{thm2.5BFF}
For any $u \in W^1L^\Phi(\omega;\mathbb R^3)$, any subset $A \subset \Omega$, and any decreasing sequence $\varepsilon \to 0$, there exists a subsequence, still denoted by $\{\varepsilon\}$ for which the $\Gamma$-limit exists and 
\[
J_{\{\varepsilon\}} (u;A)=\int_A W_{\{\varepsilon\}}(x_\alpha; \nabla_\alpha u)dx_\alpha,
\]
for a suitable Carath\'eodory function $W_{\{\varepsilon\}}$, such that $W_{\{\varepsilon\}}(x_\alpha, \cdot)$ is quasiconvex for $\mathcal L^2$-a.e. $x_\alpha \in \omega$.
\end{thm}

A related version of this theorem, dealing with oscillating profiles in the standard Sobolev setting has been obtained in \cite[Theorem 2.5]{BFF}. Indeed, we present below only the steps which really differ from the arguments therein.  

In particular, we start by a preliminary result which shows the existence of a recovery sequence for $J_{\{\varepsilon\}}(u)$ with the same lateral boundary data as $u$.

\begin{lem}\label{prelim_lemma}
	Let	$u \in W^1L^\Phi(A;\mathbb{R}^3)$ where $A$ is an open subset of $\omega$. If $v_{\varepsilon} \in W^1L^\Phi(A \times (\frac{-1}{2}, \frac{1}{2});\mathbb{R}^3)$ are such that 
	\begin{equation}\label{subseq_v_eps}
		\left\{\begin{array}{l}
			(v_{\varepsilon} - u)\longrightarrow 0 \text{ in } L^{\Phi}(A\times(-1/2,1/2); \mathbb{R}^{3}) \\
			J_{\{\varepsilon\}}(u; A)= \lim_{\varepsilon\to 0^+} E_{\varepsilon}(v_\varepsilon; A),
		\end{array}\right.
	\end{equation}
	then there exists a sequence $\{w_{\varepsilon}\}\subset W^1L^\Phi(A\times(\frac{-1}{2},\frac{1}{2}); \mathbb{R}^{3})$ which satisfies \eqref{subseq_v_eps} and is such that 
	\begin{equation*}
		w_{\varepsilon}= u \text{ in } \{(x_{\alpha}, x_{3}): \; x_{\alpha}\in  A\setminus K^\varepsilon \text{ and } |x_{3}|< \frac{1}{2}\}
	\end{equation*}
	for some compact set $K^{\varepsilon}\subset A$.
	\end{lem}
    	\begin{proof}
	
	First we get an estimate of the energy, in fact from \eqref{growthW} there exists $C>0$ such that
	\[
	\sup_{\varepsilon}
	\int_{A\times (-1/2,1/2)}
	\left(
	1
	+
	\Phi(|\nabla_\alpha v_{\varepsilon}|)
	+
	\Phi\!\left(\left|\frac{1}{\varepsilon} \nabla_3 v_{\varepsilon}\right|\right)
	\right)
	dx
	\le C.
	\]
	Then we can introduce the sets constants $K(\varepsilon)$ and the sets $A(\varepsilon)$ exactly as in \cite[equation (2.10)]{BFF}. The only difference is that now we get the Orlicz variant of \cite[eq. (2.11)]{BFF}, i.e.
\begin{equation*}
		K(\varepsilon)
		:=
		\left|\Bigl[
		\frac{1}{(\int_{A\times(-1/2,1/2)}\Phi(|v_{\varepsilon}-u|)dx)^{1/2}}
		\Bigr]\right|,
		\qquad
		M(\varepsilon):=[\sqrt{K(\varepsilon)}],
	\end{equation*}
	where $|[a]|$ is the integer part of $a$,
	
	
	
	
	\[
	A(\varepsilon)
	:=
	\left\{
	x_\alpha\in A:
	\mathrm{dist}(x_\alpha,\partial A)
	<
	\frac{M(\varepsilon)}{K(\varepsilon)}
	\right\}.
	\]
	
	

	
	\begin{equation}\label{decoupage_mince}
		\int_{(A_{i(\varepsilon)}^{\varepsilon})\times(-1/2,1/2)}
		\left(
		1+\Phi(|D_\alpha v_{\varepsilon}|)
		+\Phi\!\left(\left|\frac{1}{\varepsilon} D_3 v_{\varepsilon}\right|\right)
		\right)
		dx
		\le
		\frac{C}{M(\varepsilon)},
	\end{equation}
	for a suitable strip of the type 
    \[
	A_{i}^{\varepsilon}
	=
	\left\{
	x_\alpha\in A:
	\mathrm{dist}(x_\alpha,\partial A)
	\in
	\left[
	\frac{i}{K(\varepsilon)},
	\frac{i+1}{K(\varepsilon)}
	\right]
	\right\}, \text{ } i=0, \cdots, M(\varepsilon)-1.
	\]
    Then we consider planar cut-off functions as in \cite{BFF}
	
	$\varphi_{\overline{\varepsilon}}\in C_c^\infty(A)$

	\[
	0\le\varphi_{\varepsilon}\le1,
	\qquad
	\|D_\alpha\varphi_{\varepsilon}\|_{L^\infty}
	\le 2K(\varepsilon),
	\]
	
	\[
	\varphi_{\overline{\varepsilon}}=
	\begin{cases}
		1 & \text{if } \mathrm{dist}(x_\alpha,\partial A)
		> \frac{i(\varepsilon)+1}{K(\varepsilon)},\\[4pt]
		0 & \text{if } \mathrm{dist}(x_\alpha,\partial A)
		\le \frac{i(\varepsilon)}{K(\varepsilon)}.
	\end{cases}
	\]
	
and define
	
	\[
	w_{\varepsilon}
	:=
	\varphi_{\varepsilon} v_{\varepsilon}
	+
	(1-\varphi_{\varepsilon})u.
	\]
	
	Then
	$w_{\varepsilon}=u$ outside  $K^{\varepsilon}\times (-1/2,1/2)$, with 
    
    
    $K^{\varepsilon}:= \left\{x_{\alpha}\in A \:\ \mathrm{dist}(x_\alpha,\partial A)
	\geq \frac{i(\varepsilon)}{K(\varepsilon)} \right\}$. Moreover $w_{\varepsilon} \in W^{1}L^\Phi(A\times(-1/2,1/2); \mathbb{R}^{3})$. 
	
	\medskip

Since
	$0\le\varphi_{\varepsilon}\le1$
	and
	$v_{\varepsilon}\to u$ in $L^\Phi$ (see \eqref{subseq_v_eps}), then
	
	\begin{equation*}
		\|w_{\varepsilon}-u\|_{L^\Phi(A\times(-1/2,1/2))}
		\to0.
	\end{equation*}
	
	\medskip
	
	Since \[\nabla_\alpha w_{\varepsilon}
	=
	\varphi_{\varepsilon} \nabla_\alpha v_{\varepsilon}
	+
	(1-\varphi_{\varepsilon})\nabla_\alpha u
	+
	\nabla_\alpha\varphi_{\varepsilon}
	\otimes (v_{\varepsilon}-u),
	\]
	by $\Delta_2$ and convexity
	
	\[
	\Phi(|\nabla_\alpha w_{\varepsilon}|)
	\le
	C\big(
	\Phi(|\nabla_\alpha v_{\varepsilon}|)
	+
	\Phi(|\nabla_\alpha u|)
	+
	\Phi(|\nabla_\alpha\varphi_{\varepsilon}||v_{\varepsilon}-u|)
	\big).
	\]

	%
	%
	%
	%
	By the fact that $v_\varepsilon$ is a recovery sequence, i.e. it satisfies \eqref{subseq_v_eps}, we have 
    \[
    \begin{aligned}
			J_{\{\varepsilon\}}(u;A) \ge \\
			\limsup_{\varepsilon\to0} \int_{(A\times(-1/2,1/2))\cap [\{x_{\alpha}: \mathrm{dist}(x_\alpha,\partial A)
				> \frac{i(\varepsilon)+1}{K(\varepsilon)}\}\times (-1/2,1/2)]} 	W_{\varepsilon}\!\!\left(
			x_\alpha,x_3;
			\nabla_\alpha v_{\varepsilon} \,\middle|\, \frac{1}{\varepsilon}\nabla_3 v_{\varepsilon}
			\right)\,dx_{\alpha}dx_{3}  \\
			\geq \limsup_{\varepsilon\to0} \bigg\{ \int_{A\times(-1/2,1/2)} W_{\varepsilon}\!\left(
			x_\alpha,x_3;
			\nabla_\alpha w_{\varepsilon} \,\middle|\, \frac{1}{\varepsilon}\nabla_3 w_{\varepsilon}
			\right)\,dx_{\alpha}dx_{3}  \\
			- \beta\int_{A_{\varepsilon}\cap [\{x_{\alpha}: \mathrm{dist}(x_\alpha,\partial A)
				< \frac{i(\varepsilon)}{K(\varepsilon)}\}\times (-1/2,1/2)]} (1 + \Phi(|\nabla_{\alpha}u|)) \,dx_{\alpha}dx_{3} \\
			-\beta \int_{(A^{\varepsilon}_{i(\varepsilon)})\times (-1/2,1/2)} \left(1 + \Phi(|\nabla_\alpha v_{\varepsilon}|) + \Phi(|\frac{1}{\varepsilon}\nabla_{3}v_{\varepsilon}|)\right) \,dx_{\alpha}dx_{3} \\
			-C (K(\varepsilon)) \int_{(A^{\varepsilon}_{i(\varepsilon}))_\times(-1/2,1/2)} \Phi(|v_{\varepsilon}-u|) \,dx_{\alpha}dx_{3} \bigg\} \\
			\ge \limsup_{\varepsilon\to0}
			E_{\varepsilon}(w_{\varepsilon};A) - C\liminf_{\varepsilon\to0} \mathcal{L}^{2}(A(\varepsilon)) - C\beta\liminf_{\varepsilon\to0} \dfrac{1}{M(\varepsilon)}  \\
			-\beta\liminf_{\varepsilon\to0} (\int_{A\times(-1/2,1/2)}\Phi(|v_{\varepsilon}-u|)dx)^{1/2} \\
			= \limsup_{\varepsilon\to0}
			E_{\varepsilon}(w_{\varepsilon};A),
            \end{aligned}
            \]
	where in the fifth line we have used $\Delta_2$, exploiting the fact that the gradients of the cut-oof functions are bounded by $K(\varepsilon)$. In the sixth line we have used \eqref{decoupage_mince} to   get the last estimate.
	
	\medskip
	
	The definition of $J_{\{\varepsilon\}}$ in \eqref{freeGammaliminf}, together with the above inequality, concludes the proof.
    \end{proof}
\begin{proof}[Proof of Theorem \ref{thm2.5BFF}]
We start by observing that in the definition of $J_{\{\varepsilon\}}(u;A)$ the convergence is a metric one, since we are on bounded sets. Thus the existence of a not relabeled subsequence for which this $\Gamma$-.liminf is a $\Gamma$-limit is well known (see \cite[Theorem 8.5]{DMbook}).
 In order to prove the existence of a density $W_{\{\varepsilon\}}$ it suffices to prove that $J_{\{\varepsilon\}}$ satisfies the assumptions of \cite[Theorem 4.9]{MM}.

 In fact the assumptions on the Orlicz function $\Phi$ guarantee that $W^1L^\Phi(\Omega;\mathbb R^3)$ falls into the setting of applications of \cite[Theorem 3.3]{MM}, which constitutes a crucial step to apply also \cite[Theorem 4.9]{MM}.

 On the other hand, standard arguments (we refer to \cite[Proof of Theorem 2.5]{BFF}) guarantee that
 \begin{itemize}
 \item $J_{\{\varepsilon\}}$ is local, i.e. $J_{\{\varepsilon\}}(u;A)= J_{\{\varepsilon\}}(v;A)$ whenever $W^1L^\Phi(\omega;\mathbb R^3) \ni u=v$ $\mathcal L^N$ a.e. on $A$
 \item $J_{\{\varepsilon\}} (u;A)\leq \beta \int_A (1+ \Phi(\nabla_\alpha u))dx_\alpha$;
 \item $J_{\{\varepsilon\}}(u+c;A)= J_{\{\varepsilon\}}(u;A)$ for every $c \in \mathbb R$ and $u \in W^1L^\Phi(A;\mathbb R^3)$.
 \item $J_{\{\varepsilon\}}(\cdot,A)$ is $L^\Phi(\Omega)$-strongly lower semicontinuous. 
 \end{itemize}

 From the coercivity of $W_\varepsilon$, the translation invariance of $J_{\{\varepsilon\}}$ (stated in the third bullet above), the Poincar\'e-Wirtinger 's inequality (cf. \cite[Theorem 4.4]{CM}), the compact embedding of $W^1L^\Phi$ in $L^\Phi$ (see Theorem \ref{thmcompact}), guarantee that  $J_{\{\varepsilon\}}$ is $L^1$ lower semicontinuous.
 
 Thus it remains to prove that $J_{\{\varepsilon\}}$ is a finite nonnegative Radon measure.

 In turn, this latter property rests on proving the so-called nested subadditivity, all the other properties can be deduced exactly as in the proof of \cite[Theorem 2.5]{BFF}, hence we omit them. 

We now show that $J_{\{\varepsilon\}}$ is subadditive. For open sets $C\Subset B\subset A \subset \omega$,
\begin{equation}\label{2.18BFF}
J_{\{\varepsilon\}}(u;A) \le J_{\{\varepsilon\}}(u;B) + J_{\{\varepsilon\}}(u;A\setminus \bar C).
\end{equation}
For small $\delta>0$ we can find two open subsets $B_\delta\subset B$, $D_\delta\subset A\setminus \bar C$ with
\begin{equation}\label{2.19BFF}
\int_{A\setminus(B_\delta\cup D_\delta)} (1+\Phi(|\nabla_\alpha u|) dx_\alpha < \delta,
\end{equation}
and such that $B_\delta \cap D_\delta \not = \emptyset$. Standard arguments relying on the definition of $\Gamma$-limit and on De Giorgi's slicing techniques ensure the existence of sequences $v^{B_\delta}_{\varepsilon}$, $v^{D_\delta}_{\varepsilon}$ agreeing with $u$ on the respective lateral boundaries (see Lemma \ref{prelim_lemma}) 
which are recovery sequences for $J_{\{\varepsilon\}}$ in $B^\delta$ and $D^\delta$, respectively. Next we can define Radon measures
\begin{align*}
\lambda_{\varepsilon} := &(1+ \Phi\left(|\nabla_\alpha v^{B_\delta}_{\varepsilon}, 1/\varepsilon \nabla_3 v^{B_\delta}_{\varepsilon}|\right) + \Phi\left(|\nabla_\alpha v^{D_\delta}_{\varepsilon}, 1/\varepsilon \nabla_3 v^{D_\delta}_{\varepsilon}\right)\, \chi_{(B_\delta\cap D_\delta)_{\varepsilon}} \mathcal{L}^3.
\end{align*}
By coercivity, $\lambda_{\varepsilon}$ is  a sequence of bounded Radon measures, hence converging $ \rightharpoonup^*$, up to a not relabelled subsequence, to $\lambda$. 
Set $\widehat{\lambda}(X) := \lambda(X \times [-1/2,1/2])$ for any Borel subset 
$X \subset \omega$. 

Define, for $0 < \eta < 1$,
\[
S^{\delta}_{\eta} := 
\left\{ x \in B_{\delta} \cap D_{\delta} \;:\; 
\operatorname{dist}(x_{\alpha}, \partial B_{\delta}) = \eta 
\right\}.
\]
The family is made by pairwise disjoint open sets so that there exists  $\eta_0$ such that
\begin{equation}
\widehat{\lambda}(S^{\delta}_{\eta_0}) = 0.
\label{2.21}
\end{equation}

For $L^{\delta}_{\zeta}$, a layer of thickness $\zeta$ around $S^{\delta}_{\eta_0}$, i.e.
\[
L^{\delta}_{\zeta}
:= \{ x_\alpha \in B_\delta \cap D_\delta : 
\operatorname{dist}(x_\alpha, S^{\delta}_{\eta_0}) \le \zeta \},
\]
consider a smooth cut-off function 
$\varphi^{\delta}_{\zeta} \in C^\infty_0(\mathbb{R}^2)$ such that
\begin{equation}
\label{eq:cutoff_properties}
\begin{cases}
\|\varphi^{\delta}_{\zeta}\|_{L^\infty} \le 1, \\
\|D_\alpha \varphi^{\delta}_{\zeta}\|_{L^\infty} \le C/\zeta, \\
\varphi^{\delta}_{\zeta} = 0 
\quad \text{if } x_\alpha \in B_\delta 
\text{ and } \operatorname{dist}(x_\alpha, \partial B_\delta) \ge \eta_0 + \zeta,\\[4pt]
\varphi^{\delta}_{\zeta} = 1
\quad \text{if } x_\alpha \notin B_\delta 
\text{ or } \operatorname{dist}(x_\alpha, \partial B_\delta) \le \eta_0 - \zeta.
\end{cases}
\end{equation}

Define
\[
v_{\zeta,\varepsilon}
:=
\varphi^{\delta}_{\zeta} \, v^{D_\delta}_{\varepsilon}
+ (1 - \varphi^{\delta}_{\zeta}) \, v^{B_\delta}_{\varepsilon}
+ \chi_{A\setminus (B_\delta \cup D_\delta)_{\varepsilon}}\, u .
\]

Then $v^{\delta}_{\zeta,\varepsilon} \in W^1L^\Phi(A\times(-1/2;1/2);\mathbb{R}^3)$, and
\[
(v_{\zeta,\varepsilon} - u) \to 0 
\quad \text{in } L^\Phi(A \times(-1/2,1/2);\mathbb{R}^3).
\]

Thus, by definition of the $\Gamma$–liminf,  we have
\begin{equation}
J_{\{\varepsilon_R\}}(u;A)
\le \liminf_{\varepsilon\to 0^+} 
E_\varepsilon(v_{\zeta,\varepsilon};A).
\label{2.23}
\end{equation}
Taking into account the growth condition, \eqref{eq:cutoff_properties}, the fact that $\Phi$ satisfies the $\Delta_2$ condition,
We estimate the right-hand side:
\begin{align*}
E_\varepsilon(v_{\zeta,\varepsilon};A)
&\le
J_{\{\varepsilon_R\}}(u;B_\delta)
+ J_{\{\varepsilon_R\}}(u;D_\delta)
+ \beta \int_{A\setminus(B_\delta \cup D_\delta)} (1+\Phi(|\nabla_\alpha u|))\,dx_\alpha
\\
&\qquad
+ \beta \int_{L^{\delta}_{\zeta} \times (-1/2,1/2)}\left(1+\Phi\left(\nabla \varphi^\delta_\zeta (v^{B^\delta}_\varepsilon- v^{D^\delta}_\varepsilon)+ \varphi^\delta_\zeta (\nabla_\alpha v^{D^\delta}_\varepsilon, \frac{1}{\varepsilon}\nabla_3 v^{D^\delta_\varepsilon}) +\right.\right.\\
&\left.\left.(1-\varphi^\delta_\zeta) (\nabla_\alpha v^{B^\delta}_\varepsilon, \frac{1}{\varepsilon}\nabla_3 v^{B^\delta_\varepsilon})\right)\right) dx\\
&\leq  \beta \int_{A\setminus(B_\delta \cup D_\delta)} (1+\Phi(|\nabla_\alpha u|))\,dx_\alpha+ \\
& + \int_{L^{\delta}_{\zeta} \times (-1/2,1/2)}C(c/\zeta) \Phi(|(v^{B^\delta}_\varepsilon- v^{D^\delta}_\varepsilon)|)d x+ C
\lambda_\varepsilon(L^{\delta}_{\zeta} \times (-1/2,1/2)).
\end{align*}
where the constant $C\left(\frac{c}{\lambda}\right)$ depends on the bound on the gradient of the cut-off function and in the last estimate we have exploited the convexity of $\Phi$ and the $\Delta_2$ condition.

Since
\[
|v^{B_\delta}_\varepsilon - v^{D_\delta}_\varepsilon|
\chi_{(L^{\delta}_{\zeta}\times (-1/2,1/2))}
\le C(
|v^{B_\delta}_\varepsilon - u|\chi_{B_{\delta,\varepsilon}}
+
|v^{D_\delta}_\varepsilon - u|\chi_{D_{\delta,\varepsilon}}),
\]
the last but one term goes to zero as $\varepsilon \to 0$, while, by the definition of $\lambda_\varepsilon$
\[
\limsup_{\varepsilon\to 0^+}
\lambda_\varepsilon(L^{\delta}_{\zeta} \times (-1/2,1/2))
\le
\widehat{\lambda}(L^{\delta}_{\zeta}).
\]

As $\zeta\to 0$, by \eqref{2.21}
\[
\widehat{\lambda}(L^{\delta}_{\zeta})
\to 
\widehat{\lambda}(S^{\delta}_{\eta_0})
= 0.
\]

Letting $\zeta\to 0$ in \eqref{2.23}, and using \eqref{2.19BFF}, we obtain
\[
J_{\{\varepsilon_R\}}(u;A)
\le
\liminf_{\delta\to 0^+}
\left(
J_{\{\varepsilon_R\}}(u;B_\delta)
+
J_{\{\varepsilon_R\}}(u;D_\delta)
+ \beta\delta
\right)
\le
J_{\{\varepsilon_R\}}(u;B)
+
J_{\{\varepsilon_R\}}(u;A\setminus \bar C),
\]
which proves \eqref{2.18BFF}.

The existence of $W_{\{\varepsilon\}}$ and its properties follow arguing exactly as in the end of Step 3 and Step 4 of \cite[Theorem 2.5]{BFF} and applying  \cite[Theorem 4.9]{MM}.
\end{proof}

    \begin{thm}
    \label{intrep}
    Let $E_\varepsilon$ be as in \eqref{functionalfree} with $W_\varepsilon(x_\alpha, x_3, \xi)= W\left(\frac{x_\alpha}{\varepsilon}, x_3, \xi\right)$ with $W:\mathbb R^2 \times (-1/2,1/2)\times \mathbb R^{3\times 3}\to \mathbb R$, a Carath\'eodory function, periodic in the first variable and satisfying \eqref{growthW}.
    
    If $u \in W^{1}L^{\Phi}(\omega;\mathbb R^3)$ and if $A$ is an open set of $\omega$, then
    \[\Gamma-\lim(L^\Phi)(E_\varepsilon(u,A))=\int_A W_{\rm hom}(\nabla_\alpha u) dx_\alpha,\]
    with 
    \begin{align}\label{Whom}
   & W_{\hom}(\xi_\alpha):=  \lim_{t \rightarrow + \infty} \inf_{\varphi} \Bigg \{\frac{1}{t^2} \int_{(tQ')_{,1}} W(x_{\alpha}, x_3, \xi_\alpha + \nabla_{\alpha} \varphi | \nabla_3 \varphi) \, d x_{\alpha} d x_3: \\
& \, \varphi \in W^{1,\infty}((tQ')_{,1}; \mathbb R^3), \,\, \varphi(x_{\alpha}, x_3) = 0\,\,\, \textnormal{for every $(x_{\alpha}, x_3) \in (\partial Q')\times \left(-\frac{1}{2},\frac{1}{2}\right)$} \Bigg \} \nonumber
    \end{align}
    \end{thm}
    \begin{proof}[Proof]
    The proof of this result is entirely analogous the one of \cite[Theorem 4.2]{BFF}. In fact, it relies on Theorem \ref{thm2.5BFF} and on a comparison argument that allows to give an explicit form to $W_{\{\varepsilon\}}$.
    This result, in turn, can be proven exactly as in \cite{BFF}, just replacing the $p$-growth condition therein by the $\Phi$ one. 
    \end{proof}

    \begin{rem}
    We stress, as already emphasized in the Introduction that Theorem \ref{thm2.5BFF} can be applied to more general contexts, i.e. the same arguments apply to functionals arising in the multiscale homogenization contexts and in dimensional reduction frameworks with original and limiting dimensions
 different from $3$ and $2$, respectively. In particular Theorem \ref{intrep}, when no dimension reduction is taken into account, i.e. $W$ not depending on $x_3$ constitutes an alternative proof of the results contained in \cite{tchin4, tacha7} and in the bibliography quoted therein.  
 \end{rem}
 
	\section{Energy densities in the manifold constrained setting} 
    \label{subsect2M} 

 Given $s\in\mathcal{M}$, we consider the orthogonal projection 
\[
P_s: \mathbb R^3\to T_s(\mathcal{M}),
\]
and we define the function 
$\mathbf{P}_s: \mathbb R^{3\times 3}\to [T_s(\mathcal{M})]^3$ as
\[
\mathbf P_s(\xi):=\left[P_s(\xi_1)| P_s(\xi_2)| P_s(\xi_3)\right],
\]
for every $\xi=[\xi_1|\xi_2|\xi_3]\in\mathbb R^{3\times 3}.$ 

For a Carathéodory function $f:\mathbb R^3 \times \mathbb R^{3 \times 3}\to \mathbb R$, we set
\begin{equation}
\label{perturbedf}
    \Bar{f}(x,s,\xi):=f(x,\mathbf{P}_s(\xi))+\Phi(|\xi-\mathbf{P}_s(\xi)|).
\end{equation}
This function will play a crucial role in our subsequent analysis, since it will appear in the formulas to detect our limiting energy densities.
 By construction the function $\Bar{f}:\mathbb R^3\times\mathcal M\times\mathbb R^{3\times 3}\to\mathbb R$ is Carathéodory, i.e. it is measurable with respect to the first variable and continuous with respect to the last two variables. Moreover, if conditions (H1) and (H2) are satisfied then $\Bar{f}$ is $1$-periodic in the first variable and satisfies an analogous growth condition, i.e. there exists $\beta'>\alpha' >0$ such that
 \begin{equation}\label{pgrowth}
 \alpha'\Phi(|\xi|)\le\Bar{f}(x,s,\xi)\le \beta'(1+\Phi(|\xi|)),
 \end{equation}
for every $(s,\xi)\in \mathcal M\times\mathbb R^{3\times 3}$ and for a.e. $x\in\mathbb R^3$.

Following \cite[Proposition 2.1]{ELTZ}, we characterize the density $Tf_{\rm hom}^0$ of the $\Gamma$-limit and we will prove some of its properties.

\begin{prop}
    \label{characterization}
Let $\Phi$ be a Young function of class $\Delta_{2}\cap\nabla_{2}$.  For every  Carathéodory function $f:\mathbb R^3\times\mathbb R^{3\times 3}\to\mathbb R$ satisfying conditions {\rm (H1)} and {\rm (H2)} the following properties hold.
    \begin{itemize}
        \item[{\rm (i)}] For every $s\in\mathcal{M}$ and $\xi_\alpha\in[T_s(\mathcal{M})]^2$
        \begin{equation}
        \label{(2.3)}
        Tf_{\rm hom}^0(s, \xi_\alpha)=\Bar{f}_{\rm hom}^0(s,\xi_\alpha),
        \end{equation}
        where $Tf^0_{\rm hom}$ is defined as in \eqref{Tfhom} and
        \begin{align}
            &\Bar{f}^0_{\rm hom}(s,\xi_\alpha) =  \lim_{t \rightarrow + \infty} \inf_{\varphi} \Bigg \{\frac{1}{t^2} \int_{(tQ')_{,1}} \Bar f(x_{\alpha}, x_3,s, \xi_\alpha + \nabla_{\alpha} \varphi | \nabla_3 \varphi) \, d x_{\alpha} d x_3: \label{f0hom}\\
& \, \varphi \in W^{1,\infty}((tQ')_{,1}; \mathbb R^3), \,\, \varphi(x_{\alpha}, x_3) = 0\,\,\, \textnormal{for every $(x_{\alpha}, x_3) \in (\partial Q')\times \left(-\frac{1}{2},\frac{1}{2}\right)$} \Bigg \} \nonumber
        \end{align}
        
        \item[{\rm (ii)}] The function $Tf_{\rm hom}^0$ is tangentially quasiconvex in the second variable, i.e.
        \[
        Tf_{\rm hom}^0(s,\xi_\alpha)\le\int_{Q'} Tf_{\rm hom}^0(s, \xi_\alpha+ \nabla_\alpha \psi)dx_\alpha,
        \]
        for every $s\in\mathcal{M}, \xi_\alpha\in [T_s(\mathcal{M})]^2$, and $\psi\in W^{1,\infty}_0(Q'; T_s (\mathcal{M}))$. In particular, $Tf_{\rm hom}^0(s, \cdot)$ is rank one convex.
        \item[{\rm (iii)}] $Tf_{\rm hom}^0$  satisfies inequalities as in \eqref{pgrowth}, uniformly with respect to $s$. Moreover, there exists $C>0$ such that for every $s\in\mathcal{M}$ and $\xi_\alpha, \xi_\alpha'\in [T_s(\mathcal{M})]^2$  
        \begin{equation*}
        |Tf_{\rm hom}^0(s, \xi_\alpha) - Tf_{\rm hom}^0(s, \xi_\alpha')|\le C|\xi_\alpha-\xi_\alpha'|(1+\phi(|\xi_\alpha|+|\xi_\alpha'|)),
        \end{equation*}
       where $ \phi : [0, +\infty) \to [0, +\infty)$ is nondecreasing,
right continuous and such that
$\Phi(t) = \int_0^t \phi(s)ds$,
$\phi(0) = 0, \phi(s) > 0, s > 0, \lim_{s\to +\infty}
\phi(s) = +\infty$.

    \end{itemize}
\end{prop}

Before proving our statement, it is worth observing that for every $s \in \mathcal M$, $\Bar{f}^0_{\rm hom}(s,\cdot)$ in \eqref{f0hom}, is the $3D$-$2D$ homogenized energy density appearing in the dimension reduction problems in the unconstrained setting and it has been introduced in \cite{BFF} in the standard Sobolev setting, and recovered here in \eqref{Whom} for the case of Orlicz type growth. 
\color{black}
\begin{proof}
We start from {\rm (i)}, observing that the proof is identical to the one of \cite[(i) in Proposition 2.1]{ELTZ}.

In order to prove {\rm (ii)}, we observe that it is a consequence of Theorem \ref{intrep} applied to $W(x_\alpha, x_3, \xi):= \bar{f}(x_\alpha, x_3, s, \xi)$ in \eqref{perturbedf} with $f$  as in Theorem \ref{casopm1}.
This also yields that $\bar{f}^0_{\rm hom}(s, \cdot)$ is a quasiconvex function for every $s \in \mathcal{M}.$ 

Consequently, for any $s \in \mathcal{M}, \xi_\alpha \in [T_s(\mathcal{M})]^2$ and $\varphi \in W^{1, \infty}_0(Q', T_s(\mathcal{M})),$ it holds
\[
T f^0_{\rm hom} (s, \xi_\alpha) = \bar{f}^0_{\rm hom}(s, \xi_\alpha) \le \, \int_{Q'} \bar{f}^0_{\rm hom}(s, \xi_\alpha + \nabla_\alpha \varphi) \, dy_\alpha = \int_{Q'} T f^0_{\rm hom} (s, \xi_\alpha + \nabla_\alpha \varphi) dy_\alpha.
\]
\color{black}
This allows us to conclude that for every $s \in \mathbb R^3$,  $T f^0_{\rm hom}(s,\cdot)$ is tangentially quasiconvex. 
For the attainment of {\rm (iii)}, we observe that $T f^0_{\textnormal{hom}}(s,\cdot)$ is also rank one convex as long as \eqref{(2.3)} holds since  by {\rm (ii)} $\bar{f}^0_{\rm hom}(s, \cdot)$ is rank one convex.
\color{black}
\\
In view of assumption (H2) and the definition of $T f^0_{\textnormal{hom}}$ it is possible to show that $T f^0_{\rm hom}$  satisfies the same growth condition from above and below in the second variable uniformly with respect to the first. In fact it suffices to argue as in \cite[Section 2.4]{tacha5}. Finally, for what concerns the Lipschitz type property an argument entirely analogous to \cite[Proposition 3.2]{focar1} can be adopted taking into account the rank-one convexity of $\overline{ f}^0_{hom}$. This concludes the proof of (iii).
\end{proof}

    \color{black}

	In order to prove the upper bound in Theorem \ref{c1eq3Ma}, we need to localize the relaxed limiting energy defining
\begin{equation}
\label{Io}
    {\mathcal I}(u,A):=\inf\left\{\liminf_{\varepsilon\to 0 }I_{\varepsilon}(u_\varepsilon,A) : u_\varepsilon \to u\,\text{ in $L^\Phi(A_{,1};\mathcal M)$}\right\},
\end{equation}
for any  $u\in W^1L^\Phi(\omega; \mathcal{M})$ and $A\in\mathcal{A}(\omega)$. Moreover, we denote $I(u, \omega)$ simply by $I(u)$ for every $u\in W^1L^\Phi(\omega; \mathcal{M})$.\\
This is one of the goals of the next section in order to prove our main result.

	
	\section{Proof of Theorem \ref{casopm1}} \label{sect3M} 

    We recall that $f$ satisfies (H1) and (H2). We also recall that the candidate for the $\Gamma$-limit with respect to the strong $L^\Phi$-topology of the family of functionals $I_\varepsilon$ in \eqref{functional}, whose localization is defined in $\mathcal A(\omega)$  as  
\[
I_\varepsilon(u, A) := \left \{
\begin{array}{lll}
\!\!\!\!\!\! & \displaystyle \int_{A_{,1}} f \left (\frac{x_{\alpha}}{\varepsilon}, x_3, \nabla_\varepsilon u\right ) \, dx \qquad & \textnormal{if $u \in W^{1}L^\Phi(\Omega, \mathcal{M})$}\\[4mm]
\!\!\!\!\!\! & + \infty &\textnormal{elsewhere}
\end{array}
\right. 
\] is the functional $I$ given by \eqref{candidatepm1}.\\

\noindent
For any $A\in\mathcal{A}(\omega)$ and $u\in L^\Phi(\Omega; \mathcal{M})$, consider the functional $\mathcal I$ in \eqref{Io}.

\noindent
In order to prove the $\Gamma$- limsup inequality we introduce a suitable functional that is larger than $I$ and we prove that it is the restriction to $\mathcal{A}(\omega)$ of a Radon measure absolutely continuous with respect to $\mathcal{L}^{2}$.
Given a compact set $\mathcal{K}\subset\mathcal{M}$, and a not relabeled $\{{\varepsilon_k}\}_k$ we define for $u\in W^1L^\Phi(\omega; \mathcal{M})$ and $A \in \mathcal{A}(\omega)$

 \begin{align}\label{IKunif}
     I^{\{\varepsilon_k\}}_\mathcal{K}(u, A)&:=\inf_{\{u_k\}_k}\Bigg\{ \limsup_{k\to \infty} I_{\varepsilon_k}(u_k, A)\,:u_k \to u \textnormal{ uniformly, }  \\
     & \{\nabla_{\varepsilon_k}u_k\}_k \text{ is bounded in } L^\Phi(\Omega,\mathbb R^3),\nonumber\\
     &u_k(x)=u(x_\alpha) \text{ if } {\rm dist}(u(x_\alpha), \mathcal{K})>1 \text{ for a.e. }x\in\Omega
     \Bigg\}. \nonumber
 \end{align}

\begin{rem}
\label{abuso}
As in the unconstrained case (see Section \ref{mainsect}) we observe that, given $A\subset\mathcal{A}(\omega)$, the set
\[
V(A) := \left \{v \in W^1L^\Phi(A_{,1}; \mathcal{M}): \frac{\partial v}{\partial x_3} = 0 \,\,\, \text{a.e. on } A_{,1} \right \},
\]
is isomorphic to the Sobolev space $W^1L^\Phi(A; \mathcal{M}).$ 

Moreover, the bound on the scaled gradients and the uniform convergence in \eqref{IKunif}, guarantees that $\nabla u_k \rightharpoonup \nabla_\alpha u$.
\end{rem}


\begin{lem}[Localization]
\label{preupperbound}
    For every $u\in V(\omega),$ there exists a non-relabeled subsequence $(\varepsilon_k)_k$ such that the set function $ I^{\{\varepsilon_k\}}_\mathcal{K}(u, \cdot)$ is the restriction to $\mathcal{A}(\omega)$ of a Radon measure absolutely continuous with respect to $\mathcal{L}^{2}$.
\end{lem}
 \begin{proof}
From the $\Phi$-growth condition (H2), we obtain that, for any not relabeled subsequence $\{\varepsilon_k\}_k,$
\begin{equation*}
 I^{\{\varepsilon_k\}}_\mathcal{K}(u, A) \le \, \beta\int_A (1 + \Phi(|\nabla_\alpha u|)) \, dx; 
\end{equation*}
therefore we just need to infer the existence of a suitable subsequence $\{\varepsilon_k\}_k$ for which the trace of $ I^{\{\varepsilon_k\}}_\mathcal{K}(u, \cdot)$ is a Radon measure. This can be shown in two steps.
\\
\\
{\sc step 1:} The first thing to be proved is that the nested subadditivity for the functional  $I^{\{\varepsilon_k\}}$ 
\begin{equation}
\label{subadditivity}
 I^{\{\varepsilon_k\}}_\mathcal{K}(u, A) \le \,  I^{\{\varepsilon_k\}}_\mathcal{K}(u, B) +  I^{\{\varepsilon_k\}}_\mathcal{K}(u, A \setminus \overline{C}),
\end{equation}
holds for all $A, B, C \in \mathcal{A}(\omega)$ such that $\overline{C} \subset B \subset A.$  We observe that the construction is very similar to the one in \cite[Lemma 3.1]{ELTZ}, thus we just write the main difference.

For a given $\eta > 0$ there exist sequences $\{u_k\}_k, \{v_k\}_k \subset W^1L^\Phi(\Omega; \mathcal{M})$ such that they are almost optimal for $I^{\{\varepsilon_k\}}_{\mathcal{K}}(u, B) $ and for $I^{\{\varepsilon_k\}}_{\mathcal{K}}(u, A \setminus \overline C)$, i.e. 
\[
\left \{
\begin{array}{lll}
 \!\!\! & \displaystyle \limsup_{k \rightarrow + \infty}  I_{\varepsilon_k}(u_k, B) \le \,  I^{\{\varepsilon_k\}}_\mathcal{K}(u, B) + \eta\\[2mm]
 \!\!\! & \displaystyle \limsup_{k \rightarrow + \infty}  I_{\varepsilon_k}(v_k, A \setminus \overline{C}) \le \,  I^{\{\varepsilon_k\}}_\mathcal{K}(u, A \setminus \overline{C}) + \eta.
 \end{array}
\right.
\]
Now, we define as in \cite{ELTZ} the set
\[
\mathcal{K}' := \{s \in \mathcal{M}: {\rm dist} (s, \mathcal{K}) \le \, 1\}.
\]
Which is a compact subset of $\mathcal{M}$

Take $L := {\rm dist} (C, \partial B),$ $M \in \mathbb{N}$ and define, for every $i \in \{0, \dots, M\}$
\[
B_i := \left \{
x_\alpha \in B: {\rm dist} (x_\alpha, \partial B) > \frac{i L}{M}
\right \}.
\]
while for every $i \in \{0, \dots, M-1\}$
set 
\[
S_i := B_i \setminus \overline{B_{i+1}}.
\]
Consider finally, for every $i \in \{0, \dots, M-1\}$, $\zeta_i \in \mathcal{C}_c^{\infty}(\Omega; [0,1])$ being a cut-off function satisfying 
\[
\zeta_i(x)=\zeta_i(x_\alpha) = \left \{
\begin{array}{lll}
\!\!\! & 1 \qquad & \textnormal{in $(B_{i + 1})_{,1}$}\\[3mm]
\!\!\! & 0 \qquad & \textnormal{in $\Omega \setminus (B_{i})_{,1}$} \qquad \textnormal{and} \qquad |\nabla \zeta_i|= |\nabla_\alpha \zeta_i|\le \, \frac{2 M}{L}
\end{array}
\right.
\]

Arguing exactly as in \cite{ELTZ} we can find 
a $\delta > 0,$ $c > 0$ and a uniformly continuously differentiable mapping $\Pi: D_{\delta} \times [0,1] \rightarrow \mathcal{M},$  (with partial derivatives uniformly bounded by construction (see \cite[(2.2). Lemma 3.2 and Remark 3.3]{DFMT}, where
\[
D_{\delta} := \{(s_0, s_1) \in \mathcal{M} \times \mathcal{M}: {\rm dist}(s_0, \mathcal{K}') < \delta, \,\,\, {\rm dist}(s_1, \mathcal{K}') < \delta, \,\,\, |s_0 - s_1| < \delta\},
\]
such that  \cite[eq. (3.4) and (3.5)]{ELTZ} hold, i.e.
\begin{equation*}
\Pi(s_0, s_1, 0) = s_0, \qquad \Pi(s_0, s_1, 1) = s_1, \qquad \frac{\partial \Pi}{\partial t} (s_0, s_1, t) \le \, c \, |s_0 - s_1|,
\end{equation*}
and
\begin{equation*}
|\Pi(s_0, s_1, t) - s_0| \le \, c |s_0 - s_1|\end{equation*}

\noindent Recalling the $L^\infty$ convergence of $\{u_k\}_k$ and $\{v_k\}_k$, and the properties of the sequence, i.e. the fact that
for a.e. $x \in \Omega,$ ${\rm dist}(v_k(x), \mathcal{K}') < \delta$ when $u(x_\alpha) \in \mathcal{K}'.$ 

Thus we can define $w_{k, i} \in W^1L^\Phi(\Omega; \mathcal{M})$ as follows
\[
w_{k, i}(x) := \left \{
\begin{array}{lll}
\!\!\! & \Pi(v_k(x), u_k(x), \zeta_i(x)) \qquad & \textnormal{if $u(x) \in \mathcal{K}'$}\\[3mm]
\!\!\! & u(x_\alpha) \qquad & \textnormal{if $u(x) \notin \mathcal{K}'$.}
\end{array}
\right.
\]
By exploiting the $\Phi$-growth condition (H2), the fact that $\Phi$ satisfies $\Delta_2$,  as well as \cite[(3.4)]{ELTZ}, we are able to deduce
\[
\begin{aligned}
\int_{A_{,1}} &f \left (\frac{x_\alpha}{\varepsilon_k}, x_3, \nabla_{\varepsilon_k} w_{k,i} \right ) \, dx \\
&\le \int_{B_{,1}} f \left (\frac{x_\alpha}{\varepsilon_k}, x_3, \nabla_{\varepsilon_k} u_k \right ) \, dx + \int_{A_{,1} \setminus \overline{C}_{,1}} f \left (\frac{x_\alpha}{\varepsilon_k}, x_3, \nabla_{\varepsilon_k} v_k \right ) \, dx \\
& + C_0 \int_{(S_i)_{,1}} (1 + \Phi(|\nabla_{\varepsilon_k} u_k|) + \Phi(|\nabla_{\varepsilon_k} v_k|) + \Phi(M |u_k - v_k|) \, dx,
\end{aligned}
\]
which holds for some constants $C_0 > 0$ independent of $i, k$ and $M,$ in view of $\Delta_2$ condition and the convexity of $\Phi$. Now, if we sum up over the index $i \in \{0, \dots, M-1\}$ and we divide by $M,$ we get
\[
\begin{aligned}
\frac{1}{M} &\sum_{i = 0}^{M - 1} \int_{A_{,1}} f \left (\frac{x_\alpha}{\varepsilon_k}, x_3, \nabla_{\varepsilon_k} w_{k,i} \right ) \, dx\\
&\le \int_{B_{,1}} f \left (\frac{x_\alpha}{\varepsilon_k},  x_3, \nabla_{\varepsilon_k} u_k \right ) \, dx + \int_{A_{,1} \setminus \overline{C}_{,1}} f \left (\frac{x_\alpha}{\varepsilon_k}, x_3, \nabla_{\varepsilon_k} v_k \right ) \, dx \\
& + \frac{C_0}{M} \int_{B_{,1} \setminus \overline{C}_{,1}} (1 + \Phi(|\nabla_{\varepsilon_k} u_k|) +\Phi( |\nabla_{\varepsilon_k} v_k|) +\Phi( M |u_k - v_k|) \, dx.
\end{aligned}
\]
Therefore it is possible to find some indices $i_k \in \{0, \dots, M-1\}$ such that $\overline{w}_k := w_{k, i_k}$ satisfies
\[
\begin{aligned}
\int_{A_{,1}} &f \left (\frac{x_\alpha}{\varepsilon_k}, x_3, \nabla_{\varepsilon_k} \overline{w}_{k} \right ) \, dx\\
&\le\int_{B_{,1}} f \left (\frac{x_\alpha}{\varepsilon_k}, x_3, \nabla_{\varepsilon_k} u_k \right ) \, dx + \int_{A_{,1} \setminus \overline{C}_{,1}} f \left (\frac{x_\alpha}{\varepsilon_k}, x_3,\nabla_{\varepsilon_k} v_k \right ) \, dx \\
&+ \frac{C_0}{M} \int_{B_{,1} \setminus \overline{C}_{,1}} (1 + \Phi(|\nabla_{\varepsilon_k} u_k|) + \Phi(|\nabla_{\varepsilon_k} v_k|) +\Phi( M |u_k - v_k|) \, dx.
\end{aligned}
\]
From \cite[(3.4) and (3.5)]{ELTZ}, we get that $\overline{w}_k \rightarrow u$ uniformly; moreover $\nabla\overline{w}_k \rightharpoonup \nabla_\alpha u$ in $L^\Phi(\Omega; \mathbb{R}^3)$, $\{\nabla_{\varepsilon_k} \bar{w}_k\}_k$ is bounded in $L^\Phi(\Omega,\mathbb R^3)$, and finally $\overline{w}_k(x) = u(x_\alpha)$ whenever dist$(u(x_\alpha), \mathcal{K}) > 1$ for a.e. $x \in \Omega.$ Finally, we can conclude that
\[
\begin{aligned}
&I_{\mathcal{K}}^{\{\varepsilon_k\}}(u, A)\\
&\le \limsup_{k \rightarrow + \infty} I^{\varepsilon_k}(\overline{w}_k, A) \\
&\le \limsup_{k \rightarrow + \infty} \Bigg \{ I_{\varepsilon_k}(u_k, B) + I_{\varepsilon_k}(v_k, A \setminus \overline{C}) \\
&\qquad + \frac{C_0}{M} \int_{B_{,1} \setminus \overline{C}_{,1}} (1 +  \Phi(|\nabla_{\varepsilon_k} u_k| + |\nabla_{\varepsilon_k} v_k| + M |u_k - v_k|) \, dx  \Bigg \}\\
&\le I_{\mathcal{K}}^{\{\varepsilon_k\}}(u, B) + I_{\mathcal{K}}^{\{\varepsilon_k\}}(u, A \setminus \overline{C}) + 2 \eta + \frac{C_0}{M} \sup_{k \in \mathbb{N}} \int_{B_{,1} \setminus \overline{C}_1} (1 +\Phi( |\nabla_{\varepsilon_k} u_k| + |\nabla_{\varepsilon_k} v_k|) \, dx.
\end{aligned}
\]
Taking the limit first as $M \rightarrow + \infty$ and then as $\eta \rightarrow 0,$ we obtain the desired subadditivity property \eqref{subadditivity}.

\noindent {\sc step 2:} At this point we can find some nonnegative Radon measure $\mu \in \mathcal{M}(\omega)$ and a subsequence $\{\varepsilon_{k}\}\searrow 0^{+}$, hence we can use standard diagonalization arguments, ensuring that $I_{\mathcal{K}}^{\{\varepsilon_{k}\}}(u,A)= \mu(A)$ for any $A\in\mathcal{A}(\omega)$ and $u \in W^{1}L^{\Phi}(\omega; \mathcal{M})$. However, this step is omitted since it is entirely similar to the second step of the proof in \cite[Lemma 3.1]{ELTZ}, simply replacing the classical Sobolev space  $W^{1,p}$ by Orlicz-Sobolev spaces $W^{1}L^{\Phi}$.

\end{proof}

\noindent
Now we can prove the $\limsup$-inequality for Theorem \ref{casopm1}.

\begin{prop}[$\Gamma$-limsup]
\label{Glimsup}
For every $u\in W^{1}L^\Phi(\omega; \mathcal{M})$ it holds
\begin{equation*}
    I(u)\ge \mathcal I(u),
\end{equation*}
where $I$ and $\mathcal I$ are defined by \eqref{candidatepm1} and \eqref{Io}, respectively.
\end{prop}

\begin{proof}
Let $u \in W^{1}L^\Phi(\omega; \mathcal{M}).$ Consider $R > 0$ arbitrarily large, set
\[
\mathcal{K} := \mathcal{M} \cap B^3(0,R),
\]
where we recall that the latter denotes the closed ball centered at $0$ with radius $R$, 
and consider the sequence $\{\varepsilon_k\}_k$ given by Lemma \ref{preupperbound}. It is clear that
\[
{\mathcal I}(u) \le \, I^{\{\varepsilon_k\}}_\mathcal{K}(u, \omega).
\]
Now, we would like to show that
\begin{equation}
\label{(4.1)}
I^{\{\varepsilon_k\}}_\mathcal{K}(u, \omega) \le \int_{\omega} \left \{ \chi_R(|u|) T f_{\rm hom}^0 (u, \nabla_\alpha u) + \beta (1 - \chi_R(|u|)) (1 + \Phi(|\nabla_\alpha u|))\right \} \, dx_\alpha,
\end{equation}
where 
\[
\chi_R(t) = 
\left \{
\begin{array}{lll}
\!\!\! 1 \qquad \textnormal{for $t \le R$}\\
\\
\!\!\! 0 \qquad \textnormal{otherwise}
\end{array}
\right.
\]
In order to deduce \eqref{(4.1)}, it is enough to prove that
\[
\frac{d I^{\{\varepsilon_k\}}_\mathcal{K}(u, \cdot) }{d \mathcal{L}^2} (x_0) \le \, \chi_R(|u(x_0)|) T f_{\rm hom}^0 (u(x_0), \nabla_\alpha u(x_0)) + \beta (1 - \chi_R(|u(x_0)|) (1 + \Phi(|\nabla_\alpha u(x_0)|)),
\]
for $\mathcal{L}^2-$a.e. $x_0 \in \omega.$
\\
Let us consider $x_0 \in \omega$ to be a Lebesgue point of $u$ and $\nabla_\alpha u$, and $\Phi(|\nabla_\alpha u|)$ such that $u(x_0) \in \mathcal{M},$ $\nabla_\alpha u(x_0) \in [T_{u(x_0)}(\mathcal{M})]^2,$ and the Radon-Nykod\'ym derivative of $I^{\{\varepsilon_k\}}_\mathcal{K}(u, \cdot)$ with respect to the Lebesgue measure $\mathcal{L}^2$ exists. We observe that almost every point in $\omega$ satisfies these properties, see \cite[Corollary 2.23]{AFP}. Moreover let us set 
\[
s_0 := u(x_0) \qquad \textnormal{and} \qquad \xi_0 := \nabla_\alpha u(x_0).
\]
Assume first that $s_0 \notin \mathcal{K}.$ Then, using (H2), we obtain that 
\begin{eqnarray}
\nonumber
\frac{d I^{\{\varepsilon_k\}}_\mathcal{K}(u, \cdot) }{d \mathcal{L}^2} (x_0) &= &\lim_{\rho \rightarrow 0^+} \frac{I^{\{\varepsilon_k\}}_\mathcal{K}(u, Q'(x_0, \rho))}{\rho^2} \le \, \limsup_{\rho \rightarrow 0^+} \limsup_{k \rightarrow + \infty} \rho^{-2} I_{\varepsilon_k}(u, Q'(x_0, \rho))\nonumber \\
&\le & \, \lim_{\rho \rightarrow 0^+} \frac{\beta}{\rho^2} \int_{Q'(x_0, \rho)} (1 + \Phi(|\nabla_\alpha u|)) \, dx = \beta (1 + \Phi(|\xi_0|)),
\end{eqnarray}
which is what we would like to prove.\\

\noindent
If instead $s_0 \in \mathcal{K},$ then, the proof is identical to the one given in \cite[Proposition 3.1]{ELTZ}, just replacing the estimates obtained in term of the $p$-growth condition (see \cite[eq (3.15) and (3.16)]{ELTZ}) by the growth condition (H2) and replacing $|\cdot|^p$ by $\Phi(|\cdot|)$. Thus we omit it.

\end{proof}

    \noindent
Now we prove the $\liminf$-inequality for Theorem \ref{casopm1}.
\begin{prop}[$\Gamma$-liminf]\label{liminf}
    For every $u \in W^{1}L^\Phi(\omega; \mathcal{M})$ it holds
    \begin{equation*}
        I(u)\le \mathcal I(u),
    \end{equation*}
    where $I$ is defined in \eqref{candidatepm1}, while $\mathcal I$ is defined in \eqref{Io}.
\end{prop}
\begin{proof}[Proof]

\noindent
{\sc step 1.} Fix $u\in W^{1}L^\Phi(\omega;\mathcal{M})$. We consider a recovery sequence $\{u_n\}_n\subset W^{1}L^\Phi(\Omega; \mathcal{M})$ related to $\mathcal I(u,\omega)$. We define the sequence of non-negative Radon measure
    \[
\mu_n:=\left(\int_{(-\frac{1}{2}, \frac{1}{2})}f\left(\frac{(\cdot)_\alpha}{\varepsilon_n}, x_3, \nabla_{\varepsilon_n}u_n\right)dx_3\right)\mathcal{L}^{2}\lfloor\omega.
    \]
    Up to a subsequence, there exists a Radon measure $\mu\in \mathcal{M}(\omega)$ such that $\mu_n\rightharpoonup^*\mu$
    in $\mathcal{M}(\omega)$. 
    By Lebesgue Differentiation Theorem we can split $\mu$ into the sum of two mutually disjoint non-negative Radon measure $\mu^a$ and $\mu^s$. In particular, $\mu^a<<\mathcal{L}^{2}$, while $\mu^s$ is singular with respect to $\mathcal{L}^{2}$. By weak * convergence $\mu^a(\Omega)\le\mu(\Omega)\le \mathcal I(u)$ so we want to prove that
    \[
    \frac{d\mu}{d\mathcal{L}^{2}}(y_0)\ge Tf^0_{\rm hom}(u(y_0),\nabla_\alpha u(y_0)) \quad \text{for }\mathcal{L}^{2}-\text{ a.e. } y_0\in\omega.
    \]

    \noindent
    Let $y_0\in\omega$ be a Lebesgue point for $u$ and $\nabla_\alpha u$ and a point of of approximate differentiability for $u$, i.e. such that $u(y_0)\in \mathcal{M}$ and $\nabla_\alpha u(y_0)\in [T_{u(y_0)}(\mathcal{M})]^{2}$, and such that the Radon-Nykod\'ym derivative of $\mu$ with respect to $\mathcal{L}^{2}$ exists and it is finite. We define $s_0:=u(y_0)$ and $\xi_0:=\nabla_\alpha u(y_0)$ and we consider a vanishing sequence $\{\rho_k\}_k\subset (0,+\infty)$ such that $\mu(\partial Q' (y_0,\rho_k))=0$ for every $k\in\mathbb N$. By definition \eqref{perturbedf} of $\Bar{f}$ we get
    \begin{eqnarray*}
        \frac{d\mu}{d\mathcal{L}^{2}}(y_0)&=&\lim_{k\to +\infty}\frac{\mu(Q'(y_0,\rho_k))}{\rho_k^{2}}\nonumber\\
        &=&\lim_{k\to +\infty}\,\lim_{n\to +\infty} \frac{\mu_n(Q' (y_0,\rho_k)_{1})}{\rho_k^{2}}\nonumber\\
        &=&\lim_{k\to +\infty}\,\lim_{n\to +\infty}\int_{Q'_{,1}}f\left(\frac{y_0+\rho_ky_\alpha}{\varepsilon_n},y_3, \nabla_{\varepsilon_n}u_n(y_0+\rho_ky_\alpha, y_3)\right)dy\nonumber\\
        &=&\lim_{k\to +\infty}\,\lim_{n\to +\infty}\int_{Q'_{,1}} \Bar{f}\left(\frac{y_0+\rho_ky_\alpha}{\varepsilon_n},y_3, u_n(y_0+\rho_ky_\alpha, y_3), \nabla_{\varepsilon_n}u_n(y_0+\rho_ky_\alpha, y_3)\right)dy\nonumber\\
        &=&\lim_{k\to +\infty}\,\lim_{n\to +\infty}\int_{Q'_{,1}} \Bar{f}\left(\frac{y_0+\rho_ky_\alpha}{\varepsilon_n},y_3, s_0+v_{n,k}(y), \nabla_{\frac{\varepsilon_k}{\rho_k}}v_{n,k}(y)\right)dy,\nonumber
    \end{eqnarray*}
    with $v_{n,k}(y):=\frac{\left(u_n(y_0+\rho_ky_\alpha, y_3)-s_0\right)}{\rho_k}.$ 
\\
Since $y_0$ is a point of approximate differentiability  and $u_n\to u$ in $L^\Phi(\Omega, \mathbb R^3)$ it follows that
\[
\begin{aligned}
    &\lim_{k\to \infty} \lim_{n\to \infty} \int_{Q'_{,1}} \Phi(|v_{n,k}(y)-\xi_0y_\alpha|)dy= \\
    &\lim_{k\to \infty}\int_{Q'(y_0, \rho_k)}\frac{1}{\rho_k^2}\Phi\left(\frac{|u(y)-s_0-\xi_0(y_\alpha-y_0)|}{\rho_k}\right)dy_\alpha=0.
\end{aligned}
\]

We remark that the approximate differentiability is ensured by the results contained in \cite{AC2005}. In fact, this depends on the behaviour of $\Phi$ at $\infty$ (as in the $p$-growth case there are different properties according to the relations between $p$ and the space dimension). In one case, one applies \cite[Theorem 1.4]{AC2005}, taking also into account that $L^{\Phi_3}(\Omega)$ is continuously embedded in $L^{\Phi}(\Omega)$ ($\Phi_3$ being the {\it Sobolev conjugate} of $\Phi$, according to \cite[(1.8)]{AC2005}), and the fact that $\Phi$ can be always modified around the origin, obtaining an equivalent Young function. In the other case, one invokes  \cite[Lemma 3.2 and Theorem 1.1]{AC2005}, since our domain is bounded.

Therefore it is possible to find a diagonal sequence $\varepsilon_k := \varepsilon_{n_k} < \rho_k^2$ such that, by setting $v_k(y) := v_{n_k, k} (y)$ with $y\in\Omega$, $v_0(y_\alpha) := \xi_0 y_\alpha$ with $y_\alpha\in\omega$, then $v_k \rightarrow v_0$ in $L^\Phi(Q'_1; \mathbb{R}^{3})$ and
\begin{equation}
\label{(5.1)}
\frac{d \mu}{d \mathcal{L}^{2}}(y_0) = \lim_{k \rightarrow + \infty} \int_{Q'_{,1}} \Bar{f} \left (\frac{y_0 + \rho_k y_\alpha}{\varepsilon_k},y_3, s_0 + \rho_k v_k(y), \nabla_{\frac{\varepsilon_k}{\rho_k}} v_k(y) \right ) \, dy.
\end{equation}
At this point, we observe that $\{\nabla_{\varepsilon_k} v_k\}_k$ is bounded in $L^\Phi(Q'_{,1}; \mathbb{R}^{3\times3})$ thanks to the coercivity condition (H2). By using the Decomposition Lemma in the Sobolev-Orlicz setting \cite[Theorem 1.1]{KoZa2017}, it is possible to find a sequence $\{\bar{v}_k\}_k \subset W^{1, \infty}(Q'_{,1}; \mathbb{R}^{3})$ such that $\bar{v}_k = v_0$ on a neighborhood of $\partial (Q')_{,1},$ $\bar{v}_k \rightarrow v_0$ in $L^\Phi(Q'_1; \mathbb{R}^{3}),$ the sequence of gradients $\{\Phi(|\nabla_{\varepsilon_k} \bar{v}_k|)\}_k$ is equi-integrable, \color{black} and
\begin{eqnarray}
&& \lim_{k \rightarrow + \infty} \int_{Q'_{,1}} \Bar{f} \left (\frac{y_0 + \rho_k y_\alpha}{\varepsilon_k},y_3, s_0 + \rho_k v_k(y), \nabla_{\frac{\varepsilon_k}{\rho_k}}  v_k(y) \right ) \, dy \nonumber \\
&\ge& \, \limsup_{k \rightarrow + \infty} \int_{Q'_{,1}} \Bar{f} \left (\frac{y_0 + \rho_k y_\alpha}{\varepsilon_k},y_3, s_0 + \rho_k v_k(y), \nabla_{\frac{\varepsilon_k}{\rho_k}} \bar{v}_k(y) \right ) \, dy. \label{(5.2)}
\end{eqnarray}

\vspace{5mm}

\noindent{\sc step 2.} Let us set
\[
\frac{y_0}{\varepsilon_k} = m_k + s_k \qquad \textnormal{with} \,\,\, m_k \in \mathbb{Z}^{2} \,\,\, \textnormal{and} \,\,\, s_k \in [0,1)^{2}.
\]
We can introduce
\[
x_k := \frac{\varepsilon_k}{\rho_k} s_k \rightarrow 0 \qquad \textnormal{and} \qquad \delta_k := \frac{\varepsilon_k}{\rho_k} \rightarrow 0.
\]
We can exploit the 1-periodicity of $\Bar{f}$ with respect to its first variable, \eqref{(5.1)} and \eqref{(5.2)} to get
\begin{eqnarray}
\frac{d \mu}{d \mathcal{L}^{2}}(x_0) &\ge & \, \limsup_{k \rightarrow + \infty} \int_{Q'_{,1}}\Bar{f} \left (\frac{x_k + y_\alpha}{\delta_k},y_3, s_0 + \rho_k v_k(y), \nabla_{\delta_k}  \bar{v}_k(y) \right ) \, dy \label{(5.3)} \\
&\ge &\limsup_{k \rightarrow + \infty} \int_{x_k +Q'_{,1}}\Bar{f} \left (\frac{y_\alpha}{\delta_k},y_3, s_0 + \rho_k v_k(y_\alpha - x_k, y_3), \nabla_{\delta_k} \bar{v}_k(y_\alpha - x_k, y_3) \right ) \, dy. \nonumber 
\end{eqnarray}
At this point, we extend $v_k$ as its limit (up to fixing $v_k$ at the boundary of $\partial \omega \times (-1,1)$) and $\bar{v}_k$ by $v_0$ to  $\mathbb{R}^{2}\times (-1,1).$ \color{black} As long as $x_k \rightarrow 0,$ we deduce that $\mathcal{L}^{3} ((Q'_{,1} - x_k) \Delta Q'_{,1}) \rightarrow 0,$ and the equi-integrability of $\{\Phi(|\nabla_{\varepsilon_k} \bar{v}_k|)\}_k$ together with the growth condition (H2) implies
\begin{eqnarray*}
&& \int_{Q'_{,1} \Delta (x_k + Q'_{,1})} \Bar{f} \left (\frac{y_\alpha}{\delta_k},y_3, s_0 + \rho_k v_k(y_\alpha - x_k, y_3), \nabla_{\delta_k} \bar{v}_k(y_\alpha - x_k,y_3) \right ) \, dy \\
&\le & \, \beta' \int_{Q'_{,1} \Delta (Q'_{,1})-x_k} (1 + \Phi(|\nabla_{\delta_k} \bar{v}_k| ) \, dy \rightarrow 0.
\end{eqnarray*}
Therefore \eqref{(5.3)} entails 
\[
\frac{d \mu}{d \mathcal{L}^{2}} (y_0) \ge \, \limsup_{k \rightarrow + \infty} \int_{Q'_{,1}} \Bar{f} \left (\frac{y_\alpha}{\delta_k},y_3, s_0 + \rho_k w_k, \nabla_{\delta_k} \bar{w}_k \right ) \, dy,
\]
where $w_k(y) := v_k (y_\alpha - x_k, y_3)$ and $\bar{w}_k(y) := \bar{v}_k(y_\alpha - x_k, y_3)$ converge to $v_0$ in $L^\Phi(Q'_{,1}; \mathbb{R}^{3}),$ and $\{\Phi(|\nabla \bar{w}_k|)\}_k$ is equi-integrable as well.
\vspace{5mm}

\noindent{\sc step 3.} 
Fixed $M>0$, we denote by $E_{M,k}$ the set
\[
E_{M,k}:=\left\{ x\in Q'_{,1}: |\nabla_{\varepsilon_k} w_k|\le M \right\}.
\]
By Chebyschev inequality, and up to substitute $\Phi$ by an equivalent Orlicz function, strictly monotone, we have that $\mathcal{L}^3(Q'_1\setminus E_{M,k})\le \frac{C}{\Phi(M)}$ for some constant $C>0$. By Scorza-Dragoni Theorem, fixed $\eta>0$ there exists a compact set $K_\eta \subset \overline{Q'_{,1}}$ such that $\mathcal{L}^3(\overline{Q'_{,1}}\setminus K_\eta)\le \eta$ and such that $f:K_\eta\times\mathbb R^{3\times 3}\to \mathbb R$ is continuous. It follows that $\Bar{f}(\cdot, s, \cdot):K_\eta\times B^{3\times 3}(0,M)\to\mathbb R$ is uniformly continuous for every $s\in\mathbb R^3$. Moreover, the function
\[
\Psi_{\eta, M}(t):=\sup \left\{ |f(x,\xi)-f(x, \zeta)|: x\in K_\eta,\, \xi,\zeta\in B^{3\times 3}(0,M), |\xi-\zeta|\le t \right\},
\]
is continuous, takes value $0$ for $t=0$ and is bounded. By definition of $\Psi_{\eta, M}$ and $\mathbf P_{s}$ and \cite[Lemma 3.1]{tacha1} (see also \cite[Proposition 3.2]{focar1}) it follows that for every $x\in K_\eta$, $\xi\in B^{3\times 3}(0,M)$ and $s_1, s_2\in\mathbb R^3$ holds
\begin{align}
    |\Bar{f}(x,s_1, \xi)-\Bar{f}(x,s_2, \xi)|&\le \Psi_{\eta,M}(|\mathbf P_{s_1}(\xi)-\mathbf P_{s_2}(\xi)|) + C_M|\mathbf P_{s_1}(\xi)-\mathbf P_{s_2}(\xi)| \nonumber\\
    & \le \Psi_{\eta,M}(M|\mathbf P_{s_1}- \mathbf P{s_2}|) + C_M|\mathbf P_{s_1}- \mathbf P_{s_2}|\nonumber\\
    &:=\tilde\Psi(|\mathbf P_{s_1}- \mathbf P_{s_2}|), \nonumber
\end{align}
where $|\mathbf P_{s_1}- \mathbf P{s_2}|$ denotes the operator norm of $\mathbf P_{s_1}- \mathbf P_{s_2}$. By the previous inequality, it follows that if we denote
\[
K_\eta^{per}:= \bigcup_{j\in\mathbb Z}(j+K_\eta),
\]
then
\[
 |\Bar{f}(x,s_1, \xi)-\Bar{f}(x,s_2, \xi)|\le\tilde\Psi(|\mathbf P_{s_1}- \mathbf P{s_2}|), \nonumber
\]
for every  $x\in K_\eta^{per}$, $\xi\in B^{3\times 3}(0,M)$ and $s_1, s_2\in\mathbb R^3$.
From the previous inequality it follows that
\begin{align}
    \frac{d \mu}{d \mathcal{L}^{2}} (y_0)&\ge \limsup_{k \rightarrow + \infty} I^{\varepsilon_k}(\bar{w}_k, Q'_1)\nonumber\\
    & \ge \limsup_{k \rightarrow + \infty} \int_{E_{M,k}\cap (\delta_k K_\eta^{per})}  \Bar{f}\left (\frac{y_\alpha}{\delta_k},y_3, s_0, \nabla_{\delta_k} \bar{w}_{k} \right ) dy\nonumber\\
    & - \limsup_{k \rightarrow + \infty} C_M\int_{Q'_{,1}} \tilde\Psi_{\eta, M}(|\mathbf P_{s_0+\rho_k w_k(y)}- \mathbf P_{s_0}|) dy. \nonumber
\end{align}
Since $\tilde\Psi_{\eta, M}$ is continuous, bounded and $\tilde\Psi_{\eta, M}(0)=0$ and since $|\mathbf P_{s_0+\rho_k w_k(y)}- \mathbf P_{s_0}| \to 0$ as $k\to \infty$, then the last term in the previous inequality is also $0$. It follows that
\[
 \frac{d \mu}{d \mathcal{L}^{2}} (y_0) \ge 
 \limsup_{k \rightarrow +\infty} \int_{E_{M,k}\cap \delta_k K_\eta^{per}} \bar f\left(\frac{y_\alpha}{\delta_k}, y_3,s_0,\nabla_{\delta_k} \bar w_{k} \right )dy.
\]
From the $\Phi$-growth of $\bar f$ and from Riemann-Lebesgue Lemma we get that
\begin{align}
    &\limsup_{k \rightarrow + \infty} \int_{E_{M,k}\setminus \delta_k K_\eta^{per}}  \bar f\left(\frac{y_\alpha}{\delta_k},y_3, s_0, \nabla_{\delta_k} \bar w_{k} \right)dy
    &\le \limsup_{k \rightarrow + \infty} C(1+\Phi(M))\mathcal{L}^3(Q'_1\setminus \delta_k K_\eta^{per})\nonumber\\
    \nonumber\\
    && = C(1+\Phi(M))\mathcal{L}^3(Q'_1\setminus K_\eta)
    \le C(1+\Phi(M))\eta.\nonumber
\end{align}
From the previous inequality we deduce that
\[
\frac{d \mu}{d \mathcal{L}^{2}} (y_0) \ge \limsup_{k \rightarrow + \infty} \int_{E_{M,k}}  \Bar{f}\left (\frac{y_\alpha}{\delta_k},y_3, s_0, \nabla_{\delta_k} \bar w_{k} \right )dy - C(1+\Phi(M))\eta.
\]
Since $\eta$ is arbitrary, then for $\eta \to 0$ we get
\begin{align}
    \label{lowerboundM}
    \frac{d \mu}{d \mathcal{L}^{2}} (y_0) \ge \limsup_{k \rightarrow + \infty} \int_{E_{M,k}}  \Bar{f}\left (\frac{y_\alpha}{\delta_k},y_3, s_0, \nabla_{\delta_k} \bar w_{k} \right )dy.
\end{align}
On the other hand, by construction, $\mathcal{L}^3(Q'_{,1}\setminus E_{M,k})\to 0$ uniformly with respect to $k$ as $M\to +\infty$. Since $\{|\nabla_{\varepsilon_k}\bar w_k|^p\}_k$ is equi-integrable in $\Omega$, then from the $\Phi$-growth of $\bar f$ we get for $M \to +\infty$
\[
\sup_{k}\int_{Q'_{,1}\setminus E_{M,k}} \bar f\left(\frac{y_\alpha}{\delta_k},y_3, s_0,\nabla_{\delta_k}\bar w_k\right)dy \le \sup_{k} C \int_{Q'_{,1}\setminus E_{M,k}} \Phi\left( 1+ |\nabla_{\delta_k}\bar w_k|\right)dy \to 0. 
\]
From this limit and from \eqref{lowerboundM} we conclude that
\[
\frac{d \mu}{d \mathcal{L}^{2}} (y_0) \ge \limsup_{k \rightarrow + \infty} \int_{Q'_{,1}}  \bar f\left (\frac{y_\alpha}{\delta_k},y_3, s_0, \nabla_{\delta_k} \bar w_{k} \right )dy.
\]
Using now Theorem \ref{intrep} we get
\[
\frac{d \mu}{d \mathcal{L}^{2}} (y_0)\ge \int_{Q'} \bar f_{\rm hom}^0(u(y_0), \nabla_\alpha u(y_0))dy= \bar f_{\rm hom}^0(u(y_0), \nabla_\alpha u(y_0)).
\]
By Proposition \ref{characterization} it follows that
\[
\bar f_{\rm hom}^0(u(y_0), \nabla_\alpha u(y_0))=Tf_{\rm hom}^0(u(y_0), \nabla_\alpha u(y_0)).
\]
\end{proof}

	\begin{proof}[Proof of theorem \ref{casopm1}]
	Since $\Phi \in \Delta_{2}\cap\nabla_{2}$, $L^{\Phi}(\Omega;\mathbb{R}^{d})$ is separable (see \cite{adam}) and then there exists a subsequence $\{\varepsilon_{n_{k}}\}$ such that $\mathcal I(\cdot, \omega)$ in \eqref{Io} is the $\Gamma$-limit of $\{I_{\varepsilon_{n_{k}}}\}$ for the strong $L^{\Phi}(\Omega;\mathbb{R}^{3})$-topology. Now in view of $\Phi$-growth condition $(H_{2})$ and the closure of the pointwise constraint under strong $L^{\Phi}$-convergence, we have $\mathcal I(u) < +\infty$ if and only if $u \in W^{1}L^{\Phi}(\omega;\mathcal{M})$. Hence, as a consequence of Propositions \ref{Glimsup} and \ref{liminf}, the functionals $\{I^{\varepsilon_{n_{k}}}\}$ $\Gamma$-converge (with respect to $L^{\Phi}(\Omega;\mathbb{R}^{3})$ convergence) to the functiona $I$ defined in \eqref{candidatepm1}. 
     However, since the $\Gamma$-limit $\mathcal I$ does not depend on the extracted subsequence, in light of \cite[Proposition 8.3]{DMbook} and Definition \ref{defGammafamilies}, we get that the whole family $\{I_{\varepsilon}\}$ $\Gamma$-converges to $I$ in $L^\Phi(\Omega)$.

    \end{proof}

    We conclude by observing that the above result extends \cite[Theorem 1.1]{ELTZ}, since the $p$-growth conditions therein are a particular case of the Orlicz setting considered here.
	\begin{rem}
		Let $\Phi(t) = \frac{t^{p}}{p}$ ($p>1, \;t\geq 0$), then $\Phi \in \Delta_{2}\cap \nabla_2$ (with $\widetilde{\Phi}(t)=\frac{t^{q}}{q}$, $q=\frac{p}{p-1}$) and one has $L^{\Phi}(\Omega;\mathbb{R}^{d})\equiv L^{p}(\Omega;\mathbb{R}^{d})$, $L^{\Phi}(\Omega\times Y_{per}) \equiv L^{p}(\Omega; L^{p}_{per}(Y))$ and $W^{1}L^{\Phi}(\Omega;\mathbb{R}^{d})\equiv W^{1,p}(\Omega;\mathbb{R}^{d})$.  
        
        We omit the proofs  of the homogenization Theorem \ref{c1eq3Ma} and relaxation Theorem \ref{c1eq6Ma} since they are straightforward applications. 
	\end{rem}

	\vspace{0.3cm}

	\textbf{Acknowledgements.} Part of this work has been done during the CIMPA-ICTP visit of the second author at Department of Basic and Applied Science of Sapienza-University of Rome, whose hospitality is gratefully acknowledged. Indeed J.F. Tachago acknowledges the support of CIMPA-ICTP 'Research in Pairs' fellowship 2026,  E.~Zappale acknowledges alos the support of the project
	``Mathematical Modelling of Heterogeneous Systems (MMHS)",
	financed by the European Union - Next Generation EU,
	CUP B53D23009360006, Project Code 2022MKB7MM, PNRR M4.C2.1.1.  She is a member of the Gruppo Nazionale per l'Analisi Matematica, la Probabilit\`a e le loro Applicazioni (GNAMPA) of the Istituto Nazionale di Alta Matematica ``F.~Severi'' (INdAM). 
	

\end{document}